\documentclass[11pt]{article}
\usepackage[utf8]{inputenc}
\usepackage{amsfonts,amsthm,amsmath,amssymb,url}
\usepackage[colorlinks=true, linkcolor=mydarkblue, urlcolor=blue, citecolor=gray]{hyperref}

\usepackage{xcolor}
\usepackage{graphicx}
\usepackage{fullpage}
\usepackage{thm-restate}
\usepackage{algorithm}
\usepackage{algpseudocode}
\usepackage[T1]{fontenc}
\definecolor{mydarkblue}{rgb}{0,0.08,0.45}
\usepackage{mathtools}
\usepackage{graphicx}
\usepackage{caption}
\usepackage{subfigure}

\usepackage{bm}
\usepackage{bbm}

\usepackage[capitalize]{cleveref}
\newtheoremstyle{mystyle}
  {}
  {}
  {\itshape}
  {}
  {\bfseries}
  {.}
  { }
  {\thmname{#1}\thmnumber{ #2}\thmnote{ (#3)}}

\theoremstyle{mystyle}

\newtheorem{theorem}{Theorem}
\newtheorem{lemma}[theorem]{Lemma}
\newtheorem{corollary}[theorem]{Corollary}
\newtheorem{defn}[theorem]{Definition}

\usepackage{makecell}
\usepackage{multicol}
\usepackage{enumitem}

\usepackage[numbers]{natbib}
\title{Near Optimal Reconstruction of Spherical Harmonic Expansions}
\date{}
\author{
Amir Zandieh\\MPI-Informatics\\$\mathtt{azandieh@mpi}$-$\mathtt{inf.mpg.de}$
\and
Insu Han\\Yale University\\$\mathtt{insu.han@yale.edu}$
\and
Haim Avron\\Tel Aviv University\\$\mathtt{haimav@tauex.tau.ac.il}$
}
\usepackage{times}

\newtheorem{claim}{Claim}
\newtheorem{problm}{Problem}

\newtheorem*{theorem*}{Theorem}

\newcommand{\trace}{\mathrm{trace}}

\newcommand{\inner}[1]{\left \langle {#1} \right \rangle}
\newcommand{\bigo}{\mathcal{O}}
\newcommand{\abs}[1]{\left |#1\right|}

\def\argmin{\mathop{\rm arg~min}}

\def\0{{\bm 0}}

\def\c{{\bm c}}
\def\d{{\bm d}}

\def\f{{\bm f}}

\def\u{{\bm u}}
\def\v{{\bm v}}

\def\x{{\bm x}}
\def\y{{\bm y}}
\def\z{{\bm z}}

\def\K{{\bm K}}

\def\P{{\bm P}}
\def\Q{{\bm Q}}

\def\V{{\bm V}}

\def\Y{{\bm Y}}

\def\E{{\mathbb{E}}}

\def\Rc{\mathcal{R}}

\def\f{{\bm f}}

\def\Hc{\mathcal{H}}
\def\Kc{\mathcal{K}}
\def\RR{\mathbb{R}}

\def\SS{\mathbb{S}}

\newenvironment{proofof}[1]{\noindent{\bf Proof of #1:}}{\hfill $\blacksquare$ \par}

\makeatletter
\renewcommand{\@biblabel}[1]{[#1]\hfill}
\makeatother
\renewcommand{\citet}[1]{\citep{#1}}

\begin{document}
\maketitle

\begin{abstract}%
We propose an algorithm for robust recovery of the spherical harmonic expansion of functions defined on the $d$-dimensional unit sphere $\SS^{d-1}$ using a near-optimal number of function evaluations. 
We show that for any $f\in L^2(\SS^{d-1})$, the number of evaluations of $f$ needed to recover its degree-$q$ spherical harmonic expansion equals the dimension of the space of spherical harmonics of degree at most $q$ up to a logarithmic factor. 
Moreover, we develop a simple yet efficient algorithm to recover degree-$q$ expansion of $f$ by only evaluating the function on uniformly sampled points on $\SS^{d-1}$. Our algorithm is based on the connections between spherical harmonics and Gegenbauer polynomials and leverage score sampling methods.
Unlike the prior results on fast spherical harmonic transform, our proposed algorithm works efficiently using a nearly optimal number of samples in any dimension $d$. 
We further illustrate the empirical performance of our algorithm on numerical examples.
\end{abstract}


\section{Introduction}
Interpolation is the fundamental problem of recovering a function from a finite number of (noisy) observations. To provide accurate and reliable predictions at unobserved points we need to avoid overfitting the target function which is typically achieved through restricting our interpolant to a family of \emph{smooth or structured} functions.
In this paper we focus on interpolating square-integrable functions on the $d$-dimensional unit sphere, with low-degree spherical harmonics.
Spherical harmonics are essential in various theoretical and practical applications, including the representation of electromagnetic fields~\cite{weimer1995models}, gravitational potential~\cite{werner1997spherical}, 
cosmic microwave background radiation~\cite{kamionkowski1997statistics} and medical imaging~\cite{choi2015flash}, as well as modelling of 3D shapes in computer graphics~\cite{kazhdan2003rotation}.

We begin by observing that any function $f$ in $L^2(\SS^{d-1})$, i.e., the family of square-integrable functions defined on the sphere $\SS^{d-1}$, can be uniquely decomposed into orthogonal spherical harmonic components.
Specifically, if we denote the space of spherical harmonics of degree $\ell$ in dimension $d$ by $\Hc_{\ell}(\SS^{d-1})$,
then any function $f \in L^2(\SS^{d-1})$ has a unique orthogonal expansion $f = \sum_{\ell=0}^\infty f_\ell$ with $f_\ell \in \Hc_{\ell}(\SS^{d-1})$ (see \cref{lem:spherical-harmonic-direct-sum-decompose}).
With this observation, we aim to solve the following problem of finding the best degree $\le q$ spherical harmonic approximation to $f$ using a minimal number of samples (by essentially treating the higher order terms in $f$'s expansion as noise).
\begin{problm}[Informal Version of \cref{interpolation-problem}]\label{interpolation-problem-1}
For an unknown function $f\in L^2(\SS^{d-1})$ and an integer $q \geq 1$, efficiently (both in terms of number of samples from $f$ and computations) learn the first $q+1$ spherical harmonic components $\left\{ f_\ell \in \Hc_{\ell}( \SS^{d-1}) \right\}_{\ell=0}^q$ of $f$ which minimizes
\begin{align}
    \int_{\SS^{d-1}} \abs{\sum_{\ell=0}^q f_\ell(w) - f(w)}^2 dw.
\end{align}
\end{problm}

The \emph{angular power spectrum} $\left\| f_\ell \right\|_{\SS^{d-1}}^2$ of $f$ commonly obeys a power law decay. 
In fact, for any infinitely differentiable $f$, $\left\| f_\ell \right\|_{\SS^{d-1}}^2$ decays asymptotically faster than any rational function of $\ell$. Furthermore, for any real analytic $f$ on the sphere, $\left\| f_\ell \right\|_{\SS^{d-1}}^2$ decays exponentially.
Thus, the first $q+1$ spherical harmonic components of $f$ should well approximate $f$ for even modest $q$, and answering  \cref{interpolation-problem-1} is significantly useful for a wide range of differentiable functions.

\subsection{Our Main Results}
We reformulate \cref{interpolation-problem-1} as a least-squares regression and then solve it using techniques from randomized numerical linear algebra.  To do so, we first consider an orthonormal projection operator that maps functions in $L^2(\SS^{d-1})$ onto the space of bounded-degree spherical harmonics $\bigoplus_{\ell=0}^q \Hc_\ell(\SS^{d-1})$. Specifically, if $\Kc^{(q)}_{d}$ is the projection operator that maps any function $f$ with spherical harmonic expansion $f = \sum_{\ell=0}^\infty f_\ell$ with $f_\ell \in \Hc_\ell(\SS^{d-1})$ to $\Kc^{(q)}_{d} f = \sum_{\ell=0}^q f_\ell$, \cref{interpolation-problem-1}
can be formulated as
\begin{align*}
    \min_{g \in L^2(\SS^{d-1})} \int_{\SS^{d-1}} \abs{ \left[ \Kc^{(q)}_{d} g \right](w) - f(w)}^2 dw.
\end{align*}
However, solving this regression problem with ``continuous'' cost function is challenging. To resolve this issue, we adopt the approach of \cite{avron2019universal} which discretizes the aforementioned regression problem according to the leverage function of the operator $\Kc_d^{(q)}$.
Specifically, if we can randomly draw samples with probability proportional to the leverage function then we can recover degree-$q$ spherical harmonic expansion of $f$, i.e. $\sum_{\ell=0}^q f_\ell$, with finite number of observations.  In particular, by exploiting the connections between spherical harmonics and \emph{Zonal (Gegenbauer) Harmonics} and the fact that zonal harmonics are the reproducing kernels of $\Hc_\ell(\SS^{d-1})$ (\cref{lem:decompose-Gegenbauer}), we prove that the leverage function of the operator $\Kc_d^{(q)}$ is constant.
Thus, solving a discrete regression problem with uniformly sampled observations yields near-optimal solution to \cref{interpolation-problem-1}. Our informal results are the following.

\begin{theorem}[Informal Version of \cref{thm:efficient-regression-alg}] \label{thrm:upper-bound-informal}
Let $\beta_{q,d}$ be the dimension of spherical harmonics of degree at most $q$, i.e., $\beta_{q,d} \equiv {\rm dim}\left( \bigoplus_{\ell=0}^q \Hc_\ell(\SS^{d-1}) \right)$.
There exists an algorithm that finds a $(1+\epsilon)$-approximation to the optimal solution of \cref{interpolation-problem-1}, given $s = {\bigo}(\epsilon^{-2} \beta_{q,d} \log \beta_{q,d})$ observations of $f$ at uniformly sampled points on $\SS^{d-1}$. Moreover, the algorithm runs in $\bigo(s^2 d + s^\omega)$\footnote{$\omega<2.3727$ is the exponent of the fast matrix multiplication algorithm~\cite{williams2012multiplying}} time.
\end{theorem}
We also prove that our bound on the number of required samples is optimal up to a logarithmic factor.

\begin{theorem}[Informal Version of \cref{thm:lower-bound}]
Any (randomized) algorithm that takes $s < \beta_{q,d}$ samples on any input fails with probability greater than $9/10$, where $\beta_{q,d} \equiv {\rm dim}\left( \bigoplus_{\ell=0}^q \Hc_\ell(\SS^{d-1}) \right)$.
\end{theorem}

\subsection{Related Work}
Reconstruction of functions from small number of samples as per \cref{interpolation-problem-1} has been extensively studies in many areas of science and engineering. 
Prior results mainly consider reconstructing $1$-dimensional functions from a finite number of samples on a finite interval under the assumption that the underlying function is smooth or structured in some sense.
Notably, the influential line of work of~\cite{slepian1961prolate,landau1961prolate,landau1962prolate,xiao2001prolate} focuses on reconstructing Fourier-bandlimited functions and the work of \cite{chen2016fourier, erdelyi2020fourier} consider interpolating Fourier-sparse signals.
Recently, \citet{avron2019universal} unified the reconstruction methods in dimension $d=1$ and gave a universal sampling framework for reconstructing nearly all classes of functions with Fourier-based smoothness constraints.

One can view $1$-dimensional functions on a finite interval as function on the unit circle $\SS^1$. Thus, \cref{interpolation-problem-1} is indeed a generalization of prior works to reconstruction of functions on $\SS^{d-1}$ under the assumption that the \emph{generalized Fourier series} (\cref{lem:spherical-harmonic-direct-sum-decompose}) of the underlying function only contains bounded-degree spherical harmonics. This degree constraint on spherical harmonic expansions can be viewed as the $d$-dimensional analog of the Fourier-bandlimited function on circle $\SS^1$. 

Computing spherical harmonic expansions in dimension $d=3$ has received considerable attention in physics and applied mathematics communities. 
The algorithms for this special case of \cref{interpolation-problem-1} are known in the literature as ``fast spherical harmonic transform''~\cite{swarztrauber2000generalized, suda2002fast}. 
Most notably, \citet{rokhlin2006fast} proposed an algorithm for computing spherical harmonic expansion of degree $\le q$ to precision $\epsilon$ using $\bigo(\beta_{q,3})$ samples and $\bigo(\beta_{q,3}\log \beta_{q,3} \cdot \log (1/\epsilon))$ time. 
These fast algorithms were developed based on the fast Fourier transform and fast associated Legendre transform and require access to a (well-conditioned) orthogonal basis of $\Hc_\ell(\SS^{d-1})$, which happened to be the associated Legendre polynomials when $d=3$. 
However, it is in general intractable to compute an orthogonal basis for spherical harmonics~\cite{minh2006mercer}, so unlike our \cref{thrm:upper-bound-informal}, it is inefficient to extend these prior results to higher $d$.


\section{Mathematical Preliminaries}
We denote by $\SS^{d-1}$ the unit sphere in $d$ dimension.
We use $|\SS^{d-1}|=\frac{2\pi^{d/2}}{\Gamma(d/2)}$ to denote the surface area of the sphere $\SS^{d-1}$ and $\mathcal{U}(\SS^{d-1})$ to denote the uniform probability distribution on $\SS^{d-1}$.
We denote by $L^2\left(\SS^{d-1}\right)$ the set of all square-integrable real-valued functions on the sphere $\SS^{d-1}$.
Furthermore, for any $f,g \in L^2\left( \SS^{d-1} \right)$ we use the following definition of inner product on the unit sphere,
\begin{equation}\label{eq:inner-prod-def}
    \langle f, g \rangle_{\SS^{d-1}} \coloneqq \int_{\SS^{d-1}} f(w) g(w) dw = \left| \SS^{d-1} \right| \cdot \E_{w \sim \mathcal{U}(\SS^{\d-1})} [ f(w) g(w) ].
\end{equation}
The function space $L^2\left( \SS^{d-1} \right)$ is complete with respect to the norm induced by the inner product, i.e. $\|f\|_{\SS^{d-1}} \coloneqq \sqrt{\langle f, f \rangle_{\SS^{d-1}}}$, so $L^2\left( \SS^{d-1} \right)$ is a \emph{Hilbert space}. 

We often use the term \emph{quasi-matrix} which informally defines as a ``matrix'' in which one dimension is finite while the other is infinite. A quasi-matrix can be {\em tall} (or {\em wide}) in which there is a finite number of columns (or rows) where each one is a functional operator.  For more details and a formal definition, see~\cite{shustin2021semi}.

Our results are profoundly related to the \emph{Spherical Harmonics}, which are special functions defined on $\SS^{d-1}$ and are often employed in solving partial differential equations. {\em Harmonics} are solutions to the Laplace's equation on some domain. 
Spherical harmonics are the harmonics on a spherical domain, i.e. the solution of Laplace's equation in  spherical domains. Formally,

\begin{defn}[Spherical Harmonics]
For integers $\ell \ge 0$ and $d \ge 1$, let $\mathcal{P}_{\ell}(d)$ be the space of degree-$\ell$ homogeneous polynomials with $d$ variables and real coefficients.
Let $\Hc_\ell(d)$ denote the space of degree-$\ell$ harmonic polynomials in dimension $d$, i.e., homogeneous polynomial solutions of Laplace's equation:
\[ \Hc_\ell(d) := \{ P \in \mathcal{P}_{\ell}(d) : \Delta P = 0 \}, \]
where $\Delta = \frac{\partial^2}{\partial x_1^2} +  \cdots +\frac{\partial^2}{\partial x_d^2} $ is the Laplace operator on $\RR^{d}$.
Finally, let $\Hc_\ell\left(\SS^{d-1}\right)$ be the space of (real) \emph{Spherical Harmonics} of order $\ell$ in dimension $d$, i.e. restrictions of harmonic polynomials in $\Hc_\ell(d)$ to the sphere $\SS^{d-1}$.
The dimension of this space, $\alpha_{\ell,d} \equiv {\rm dim}\left( \Hc_{\ell}\left( \SS^{d-1}\right) \right)$, is
\[ \alpha_{0,d}=1, ~~~ \alpha_{1,d} = d, ~~~ \alpha_{\ell,d} = {d+\ell-1 \choose \ell} - {d+\ell-3 \choose \ell-2} ~~~\text{ for }\ell \ge 2.\]
\end{defn}

\subsection{Gegenbauer Polynomials} The {\em Gegenbauer} (a.k.a. \emph{ultraspherical}) {\em polynomial} of degree $\ell\ge 0$ in dimension $d\ge2$ is given by
\begin{align}
    P_{d}^\ell(t) := \sum_{j=0}^{\lfloor \ell/2 \rfloor} c_j \cdot t^{\ell-2j} \cdot (1 - t^2)^j,
\end{align}
where $c_0 = 1$ and $c_{j+1} = - \frac{(\ell - 2j)(\ell - 2j - 1)}{2(j+1)(d-1 + 2j)} \cdot c_j$ for $j = 0,1, \ldots, \lfloor \ell/2 \rfloor-1$. These polynomials satisfy the orthogonality properties on the interval $[-1,1]$ with respect to the measure $(1-t^2)^{\frac{d-3}{2}}$, i.e.,
\begin{equation} \label{Gegenbauer-Poly-inner-prod}
\int_{-1}^1 P_{d}^\ell(t) \cdot P_{d}^{\ell'}(t) \cdot (1-t^2)^{\frac{d-3}{2}} \, dt = 
\frac{\left| \SS^{d-1} \right|}{ \alpha_{\ell,d} \cdot \left| \SS^{d-2} \right|} \cdot \mathbbm{1}_{\{\ell = \ell'\}}.
\end{equation}

\paragraph{Zonal Harmonics.}
The Gegenbauer polynomials naturally provide positive definite dot-product kernels on $\SS^{d-1}$ known as \emph{Zonal Harmonics}, which are closely related to the spherical harmonics.
The following reproducing property of zonal harmonics plays a crucial role in our analysis.
\begin{restatable}[Reproducing Property of Zonal Harmonics]{lemma}{reproducingpropertygegenbauerkernels}
\label{claim-gegen-kernel-properties}
	Let $P_d^\ell(\cdot)$ be the Gengenbauer polynomial of degree $\ell$ in dimension $d$. For any $x,y \in \SS^{d-1}$:
	\begin{align*}
		P_d^\ell(\langle x, y \rangle) = \alpha_{\ell,d} \cdot \mathbb{E}_{w\sim \mathcal{U}(\SS^{d-1})} \left[ P_d^\ell\left( \langle x , w \rangle \right) P_d^\ell\left( \langle y , w \rangle \right) \right] ,
	\end{align*}
	Furthermore, for any $\ell' \neq \ell$: 
	\begin{align*}
		\mathbb{E}_{w\sim \mathcal{U}(\SS^{d-1})} \left[ P_d^\ell\left( \langle x , w \rangle \right) \cdot P_d^{\ell'}\left( \langle y , w \rangle \right) \right] = 0.
	\end{align*}
\end{restatable}

The proof, like most proofs, is deferred to the appendix. 
The following very useful fact (a.k.a. Addition Theorem) connects Gegenbauer polynomials and spherical harmonics.
\begin{restatable}[Addition Theorem]{theorem}{additiveformula}\label{lem:additive-formula}
For every integer $\ell \ge 0$, if $\left\{ y^\ell_1, y^\ell_2, \ldots, y^\ell_{\alpha_{\ell,d}} \right\}$ is an orthonormal basis for $\Hc_{\ell}\left( \SS^{d-1} \right)$, then for any $\sigma, w \in \SS^{d-1}$ we have
\[ \frac{\alpha_{\ell,d}}{|\SS^{d-1}|} \cdot P_{d}^{\ell}\left( \langle \sigma , w \rangle \right) = \sum_{j=1}^{\alpha_{\ell,d}} y^\ell_j(\sigma) \cdot y^\ell_j(w). \]
\end{restatable}


\section{Reconstruction of $L^2\left(\SS^{d-1}\right)$ Functions via Spherical Harmonics}

In this section we show how to reconstruct any function $f \in L^2\left(\SS^{d-1}\right)$ from optimal number of samples via the spherical harmonics. 
We begin with showing that the  spherical harmonics form a complete set of orthonormal functions and thus form an orthonormal basis of the Hilbert space of square-integrable functions on the surface of the sphere $\SS^{d-1}$. This is analogous to periodic functions, viewed as functions defined on a circle, which can be expressed as a linear combination of circular functions (sines and cosines) via the Fourier series.

\begin{restatable}[Direct Sum Decomposition of $L^2(\SS^{d-1})$]{lemma}{directsumdecompose}\label{lem:spherical-harmonic-direct-sum-decompose}
The family of spaces $\Hc_{\ell}\left( \SS^{d-1} \right)$ yields a Hilbert space direct sum decomposition
$L^2\left( \SS^{d-1} \right) = \bigoplus_{\ell=0}^{\infty} \Hc_{\ell}\left( \SS^{d-1} \right)$: the summands are closed and pairwise orthogonal, and that every $f \in L^2\left( \SS^{d-1} \right)$ is the sum of a converging series (in the sense of mean-square convergence with the $L^2$-norm defined in \cref{eq:inner-prod-def}),
\[f = \sum_{\ell=0}^\infty f_\ell,\]
where $f_\ell \in \Hc_{\ell}\left( \SS^{d-1} \right)$ are uniquely determined functions. Furthermore, given any orthonormal basis $\left\{ y^\ell_1, y^\ell_2, \ldots, y^\ell_{\alpha_{\ell,d}} \right\}$ of $\Hc_{\ell}\left( \SS^{d-1} \right)$ we have
$f_\ell = \sum_{j=1}^{\alpha_{\ell,d}} \langle f, y^\ell_j \rangle_{\SS^{d-1}} \cdot y^\ell_j$.

\end{restatable}

The series expansion in \cref{lem:spherical-harmonic-direct-sum-decompose} is the analog of the Fourier expansion of periodic functions, and is known as $f$'s ``\emph{generalized Fourier series}''~\cite{pennell1930generalized} with respect to the Hilbert basis $\left\{ y^\ell_j : j \in [\alpha_{\ell,d}], \ell \ge 0 \right\}$.
We remark that it is in general intractable to compute an orthogonal basis for the space of spherical harmonics~\cite{minh2006mercer}, which renders the generalized Fourier series expansion in \cref{lem:spherical-harmonic-direct-sum-decompose} primarily existential. While finding the generalized Fourier expansion of a function $f \in L^2\left( \SS^{d-1} \right)$ is computationally intractable, our goal is to answer the next fundamental question, which is about finding the projection of a function $f$ onto the space of spherical harmonics, i.e., the $f_\ell$'s in \cref{lem:spherical-harmonic-direct-sum-decompose}. Concretely,
we seek to solve the following problem.

\begin{problm}\label{interpolation-problem}
For an integer $q\ge 0$ and a given input function $f \in L^2\left( \SS^{d-1} \right)$ whose decomposition over the Hilbert sum $\bigoplus_{\ell=0}^{\infty} \Hc_{\ell}\left( \SS^{d-1} \right)$ is $f = \sum_{\ell=0}^\infty f_\ell$ as per \cref{lem:spherical-harmonic-direct-sum-decompose},
let us define the low-degree expansion of this function as $f^{(q)} := \sum_{\ell=0}^q f_\ell$.
How efficiently can we learn $f^{(q)} \in \bigoplus_{\ell=0}^{q}\Hc_{\ell}\left( \SS^{d-1} \right)$? More precisely, we want to find a set $\{ w_1, w_2, \ldots, w_s \} \subseteq \SS^{d-1}$ with minimal cardinality $s$ along with an efficient algorithm that given samples $\{f(w_i)\}_{i=1}^s$ can interpolate $f(\cdot)$ with a function $\tilde{f}^{(q)} \in \bigoplus_{\ell=0}^{q} \Hc_{\ell}\left( \SS^{d-1} \right)$ such that: 
\[ \left\| \tilde{f}^{(q)} - f^{(q)} \right\|_{\SS^{d-1}}^2 \le \epsilon \cdot \left\| f^{(q)} - f \right\|_{\SS^{d-1}}^2.\] 
\end{problm}
To see why learning the low-degree expansion of a function $f$ in \cref{interpolation-problem} makes sense, note that the \emph{angular power spectrum} of $f$ commonly obeys a power law decay of the form $\left\| f_\ell \right\|_{\SS^{d-1}}^2 \le \bigo (\ell^{-s})$, for some $s > 0$, depending on the order of differentiability of $f$. 
In particular, the Sobolev inequalities imply that for any infinitely differentiable $f$, $\left\| f_\ell \right\|_{\SS^{d-1}}^2$ decays faster than any rational function of $\ell$ as $\ell \to \infty$. 
Thus, $f^{(q)}$ should well approximate $f$ for even modest $q$, and learning the low-degree expansion $f^{(q)}$ is extremely useful for a wide range of differentiable functions.

For ease of notation, we denote the Hilbert space of spherical harmonics of degree at most $q$ by $\Hc^{(q)}\left( \SS^{d-1} \right) := \bigoplus_{\ell=0}^{q}\Hc_{\ell}\left( \SS^{d-1} \right)$.
To answer \cref{interpolation-problem} we exploit the close connection between the spherical harmonics and Gengenbauer polynomials, and in particular the fact that the zonal harmonics are the reproducing kernels of the Hilbert spaces $\Hc_{\ell}\left( \SS^{d-1} \right)$.
\begin{restatable}[Reproducing Kernel of $\Hc_{\ell}\left( \SS^{d-1} \right)$]{lemma}{projectionltofunctionsintohilberpspacesphericharmonic}\label{lem:decompose-Gegenbauer}
For every $f \in L^2\left( \SS^{d-1} \right)$, if $f = \sum_{\ell=0}^\infty f_\ell$ is the unique decomposition of $f$ over $\bigoplus_{\ell=0}^{\infty} \Hc_{\ell}\left( \SS^{d-1} \right)$ as per \cref{lem:spherical-harmonic-direct-sum-decompose}, then $f_\ell$ is given by
\[ f_\ell(\sigma) = \alpha_{\ell,d} \cdot \E_{w \sim \mathcal{U} (\SS^{d-1}) } \left[ f(w) P_{d}^{\ell}\left( \langle \sigma , w \rangle \right) \right] ~~~\text{ for }\sigma \in \SS^{d-1}. \]
\end{restatable}

Now we define a kernel operator, based on the low-degree Gegenbauer polynomials, which projects functions onto their low-degree spherical harmonic expansion.
\begin{defn}[Projection Operator onto $\Hc^{(q)}(\SS^{d-1})$]\label{eq:kernel-operator-def}
For any integers $q \ge 0$ and $d\ge 2$, define the kernel operator $\Kc_{d}^{(q)} : L^2\left( \SS^{d-1} \right) \to L^2\left( \SS^{d-1} \right)$ as follows: for $f \in L^2\left( \SS^{d-1} \right)$ and $\sigma 
\in \SS^{d-1}$,
\begin{equation}
    \left[\Kc^{(q)}_{d} f \right](\sigma) := \sum_{\ell =0}^q \frac{\alpha_{\ell,d}}{|\SS^{d-1}|} \left< f, P_{d}^{\ell}\left( \langle \sigma , \cdot \rangle \right) \right>_{\SS^{d-1}} = \sum_{\ell =0}^q \alpha_{\ell,d} \cdot \E_{w \sim \mathcal{U} (\SS^{d-1}) } \left[ f(w) P_{d}^{\ell}\left( \langle \sigma , w \rangle \right) \right].
\end{equation}
This is an integral operator with kernel function $k_{q,d}(\sigma,w) := \sum_{\ell =0}^q \frac{\alpha_{\ell,d}}{|\SS^{d-1}|} \cdot P_{d}^{\ell}\left( \langle \sigma , w \rangle \right)$.
\end{defn}
Now note that the operator $\Kc^{(q)}_{d}$ is self-adjoint and positive semi-definite.
Moreover, using the reproducing property of this kernel we can establish that $\Kc^{(q)}_{d}$ is a projection operator.
\begin{restatable}{claim}{kerneloperatorisprojectionoperator}\label{claim:k-square-eqaul-k} The operator $\Kc^{(q)}_{d}$ defined in \cref{eq:kernel-operator-def} satisfies the property $\left(\Kc^{(q)}_{d}\right)^2 = \Kc^{(q)}_{d}$.
\end{restatable}

Furthermore, by the addition theorem (\cref{lem:additive-formula}), the operator $\Kc^{(q)}_{d}$ is trace-class (i.e., the trace is finite and independent of the choice of basis) because
\begin{align}
    \trace\left( \Kc^{(q)}_{d} \right) &= \int_{\SS^{d-1}} k_{q,d}(w,w) \, dw\nonumber\\ 
    &= \sum_{\ell=0}^q \frac{\alpha_{\ell,d}}{|\SS^{d-1}|} \cdot \int_{\SS^{d-1}} P_{d}^{\ell}\left( \langle w , w \rangle \right)  \, dw\nonumber\\
    &= \sum_{\ell=0}^q \alpha_{\ell,d} = {d+q-1 \choose q} + {d+q-2 \choose q-1} - 1. \label{eq:operator-K-trace-class}
\end{align}

By combining \cref{lem:additive-formula,lem:spherical-harmonic-direct-sum-decompose}, and using the definition of the projection operator $\Kc^{(q)}_d$, it follows that for any function $f \in L^2\left( \SS^{d-1} \right)$ with Hilbert sum decomposition $f = \sum_{\ell=0}^\infty f_\ell$, the low-degree component $f^{(q)} = \sum_{\ell=0}^q f_\ell \in\Hc^{(q)}\left( \SS^{d-1} \right)$ can be computed as $f^{(q)} = \Kc^{(q)}_{d} f$.
Equivalently, in order to learn $f^{(q)}$, it suffices to solve the following least-squares regression problem,
\begin{equation}\label{eq:regression-problem}
    \min_{g \in L^2\left( \SS^{d-1} \right)} \left\| \Kc^{(q)}_{d} g - f \right\|_{\SS^{d-1}}^2.
\end{equation}
If $g^*$ is an optimal solution to the above regression problem then
$f^{(q)} = \Kc^{(q)}_{d} g^*$.
In the next claim we show that solving the least squares problem in \cref{eq:regression-problem}, even to a coarse approximation, is sufficient to solve our interpolation problem (i.e., \cref{interpolation-problem}):
\begin{restatable}{claim}{approximateregressionsolvesinterpolation}\label{claim-interpolation-via-least-squares}
For any function $f \in L^2\left( \SS^{d-1} \right)$, any integer $q \ge 0$, and any $C\ge1$, if $\tilde{g} \in L^2\left( \SS^{d-1} \right)$ is a function that satisfies,
\[ \left\| \Kc^{(q)}_{d} \tilde{g} - f \right\|_{\SS^{d-1}}^2 \le C \cdot \min_{g \in L^2\left( \SS^{d-1} \right)} \left\| \Kc^{(q)}_{d} g - f \right\|_{\SS^{d-1}}^2, \]
and if we let $f^{(q)} \coloneqq \Kc^{(q)}_{d} f$, where $\Kc^{(q)}_{d}$ is defined as per \cref{eq:kernel-operator-def}, then the following holds
\[ \left\| \Kc^{(q)}_{d} \tilde{g} - f^{(q)} \right\|_{\SS^{d-1}}^2 \le (C-1) \cdot \left\| f^{(q)} - f \right\|_{\SS^{d-1}}^2. \]
\end{restatable}

\cref{claim-interpolation-via-least-squares} shows that solving the regression problem in \cref{eq:regression-problem} approximately provides a solution to our spherical harmonics interpolation problem (\cref{interpolation-problem}). But how can we solve this least-squares problem efficiently? 
Not only does the problem involve a possibly infinite dimensional parameter vector $g$,
but the objective function also involves the continuous domain on the surface of $\SS^{d-1}$.

\subsection{Randomized Discretization via Leverage Function Sampling} \label{sec:discretization}
We solve the continuous regression in \cref{eq:regression-problem} by randomly discretizing the sphere $\SS^{d-1}$, thereby reducing our problem to a regression on a finite set of points $w_1,w_2, \ldots, w_s \in \SS^{d-1}$. 
In particular, we propose to sample points on $\SS^{d-1}$ with probability proportional to the so-called \emph{leverage function}, a specific distribution that has been widely applied in randomized algorithms for linear algebra problems on discrete matrices~\cite{li2013iterative}. We start with the definition of the leverage function:

\begin{defn}[Leverage Function]\label{def:leverage-function}
For integers $q \ge 0$ and $d>0$, we define the leverage function of the operator $\Kc^{(q)}_{d}$  (see \cref{eq:kernel-operator-def}) for every $w \in \SS^{d-1}$ as follows,
\begin{align}
    \tau_{q}(w) := \max_{{g \in L^2\left( \SS^{d-1} \right) } }  \left\| \Kc^{(q)}_{d} g \right\|_{\SS^{d-1}}^{-2} \cdot {\left| \left[ \Kc^{(q)}_{d} g\right](w) \right|^2} 
\end{align}
\end{defn}
Intuitively, $\tau_q(w)$ is an upper bound of how much a function that is spanned by the eigenfunctions of the operator $\Kc^{(q)}_{d}$ can
``blow up'' at $w$.
The larger the leverage function $\tau_q(w)$ implies the higher the probability we will be required to  sample $w$. 
This ensures that our sample points well reflect any possibly significant components, or ``spikes'', of the function.
Ultimately, the integral $\int_{\SS^{d-1}}\tau_q(w)\, dw$ determines how many samples
we require to solve the regression problem \cref{eq:regression-problem} to a given accuracy. It is an already known fact that the leverage function integrates to the rank of the operator $\Kc^{(q)}_{d}$ (which turns out to be equal to the dimensionality of the Hilbert space $\Hc^{(q)}(\SS^{d-1})$). 
This will ultimately
allow us to achieve a $\widetilde{\bigo}(\sum_{\ell=0}^q \alpha_{\ell,d})$ sample complexity bound for solving the interpolation \cref{interpolation-problem}. 
To express the leverage function as a closed form, we make use of the following lemma that gives a useful alternative characterization of the leverage function.
\begin{restatable}[Min Characterization of the Leverage Function]{lemma}{mincharacterleveragefunction}\label{lem:min-char-leverage-function}
For any $w \in \SS^{d-1}$, let $\tau_q(w)$ be the leverage function (\cref{def:leverage-function}) and define $\phi_w \in L^2(\SS^{d-1})$ by $\phi_w(\sigma) \equiv \sum_{\ell=0}^q \frac{\alpha_{\ell,d}}{|\SS^{d-1}|}  P_{d}^{\ell}\left( \langle \sigma , w \rangle \right)$. We have the following minimization characterization of the leverage function:
\begin{equation}\label{eq:lev-score-min-char}
    \tau_q(w) = \left\{\min_{g \in L^2(\SS^{d-1})} \|g\|_{\SS^{d-1}}^2, ~~~~\text{s.t.}~~~ \Kc^{(q)}_{d} g = \phi_w \right\}.
\end{equation}
\end{restatable}
We prove this lemma in \cref{apndx-proof-leverage-score-regression}.
Using the minimization and maximization characterizations of the leverage function we can find upper and lower bounds on this function. Surprisingly, in this case the upper and lower bounds match, so we actually have an exact value for the leverage function. 
\begin{restatable}[Leverage Function is Constant]{lemma}{leveragescoreisconstantlemma}\label{lem:leverage-function-value}
The leverage function given in \cref{def:leverage-function} is equal to
$\tau_q (w) = \sum_{\ell=0}^q \frac{\alpha_{\ell,d}}{|\SS^{d-1}|}$   for every $w \in \SS^{d-1}$.
\end{restatable}
\begin{proof}
First we prove that $\tau_q(w) \le \sum_{\ell=0}^q \frac{\alpha_{\ell,d}}{|\SS^{d-1}|}$ using the min-characterization. If we let $\phi_w \in L^2(\SS^{d-1})$ be defined as $\phi_w(\sigma) := \sum_{\ell=0}^q \frac{\alpha_{\ell,d}}{|\SS^{d-1}|}  P_{d}^{\ell}\left( \langle \sigma , w \rangle \right)$, then
by \cref{eq:kernel-operator-def}, for every $\sigma \in \SS^{d-1}$ we can write,
\begin{align}
    \left[\Kc^{(q)}_{d} \phi_w\right](\sigma) &= \sum_{\ell=0}^q \alpha_{\ell,d} \cdot \E_{v \sim \mathcal{U}(\SS^{d-1}) } \left[P_{d}^{\ell}\left( \langle \sigma , v \rangle \right) \cdot \phi_w(v) \right]\nonumber\\
    &= \sum_{\ell=0}^q \sum_{\ell'=0}^q \frac{\alpha_{\ell,d} \alpha_{\ell',d}}{|\SS^{d-1}|} \cdot \E_{v \sim \mathcal{U}(\SS^{d-1}) } \left[ P_{d}^{\ell}\left( \langle \sigma , v \rangle \right) \cdot P_{d}^{\ell'}\left( \langle v , w \rangle \right) \right] \nonumber\\
    &= \sum_{\ell=0}^q \frac{\alpha_{\ell,d}}{|\SS^{d-1}|} P_{d}^{\ell}\left( \langle \sigma , w \rangle \right) = \phi_w(\sigma), \label{eq:K-times-phi}
\end{align}
where the third line above follows from \cref{claim-gegen-kernel-properties}. Therefore, the test function $g := \phi_w$ satisfies the constraint of the minimization in \cref{eq:lev-score-min-char}, i.e., $\Kc^{(q)}_{d} g = \phi_w$. Thus, \cref{lem:min-char-leverage-function} implies that,
\[ \tau_q(w) \le \|g\|_{\SS^{d-1}}^2 = \|\phi_w\|_{\SS^{d-1}}^2 = \sum_{\ell=0}^q \frac{\alpha_{\ell,d}}{|\SS^{d-1}|},\]
where the equality above follows from \cref{claim-gegen-kernel-properties} along with \cref{eq:inner-prod-def}. This establishes the upper bound on the leverage function that we sought to prove. 

Now, using the maximization characterization of the leverage function in \cref{def:leverage-function}, we prove that $\tau_q(w) \ge \sum_{\ell=0}^q \frac{\alpha_{\ell,d}}{|\SS^{d-1}|}$. Again, we consider the same test function $g = \phi_w$ and write,
\begin{align*}
    {\left\| \Kc^{(q)}_{d} \phi_w \right\|_{\SS^{d-1}}^{-2}} \cdot {\left| \left[ \Kc^{(q)}_{d} \phi_w\right](w) \right|^2} &= \frac{\left| \phi_w(w) \right|^2}{\left\| \phi_w \right\|_{\SS^{d-1}}^2}\\
    &= \frac{ \left| \sum_{\ell=0}^q\frac{\alpha_{\ell,d}}{|\SS^{d-1}|} P_{d}^{\ell}\left( \langle w , w \rangle \right) \right|^2 }{\sum_{\ell=0}^q \frac{\alpha_{\ell,d}}{|\SS^{d-1}|}}\\
    &= \frac{ \left| \sum_{\ell=0}^q\frac{\alpha_{\ell,d}}{|\SS^{d-1}|} P_{d}^{\ell}(1) \right|^2 }{\sum_{\ell=0}^q \frac{\alpha_{\ell,d}}{|\SS^{d-1}|}}= \sum_{\ell=0}^q \frac{\alpha_{\ell,d}}{|\SS^{d-1}|},
\end{align*}
where the first and second line above follow from \cref{eq:K-times-phi} and \cref{claim-gegen-kernel-properties}, respectively. Therefore, the max characterization of the leverage function in \cref{def:leverage-function} implies that,
\[ \tau_q(w) \ge {\left\| \Kc^{(q)}_{d} \phi_w \right\|_{\SS^{d-1}}^{-2}} \cdot {\left| \left[ \Kc^{(q)}_{d} \phi_w\right](w) \right|^2} = \sum_{\ell=0}^q \frac{\alpha_{\ell,d}}{|\SS^{d-1}|}. \]
This completes the proof of \cref{lem:leverage-function-value} and establishes that $\tau_q(w)$ is uniformly equal to $\sum_{\ell=0}^q \frac{\alpha_{\ell,d}}{|\SS^{d-1}|}$.

\end{proof}

The integral of the leverage function, which determines the total samples needed to solve our least-squares regression, is therefore equal to the dimensionality of the Hilbert space $\Hc^{(q)}(\SS^{d-1})$.
\begin{corollary}
The leverage function defined in \cref{def:leverage-function} integrates to the dimensionality of the Hilbert space $\Hc^{(q)}(\SS^{d-1})$, which we denote by $\beta_{q,d}$, i.e., 
    \[ \int_{\SS^{d-1}}\tau_q(w)\, dw = {\rm dim}\left( \Hc^{(q)}(\SS^{d-1}) \right) = \sum_{\ell=0}^q \alpha_{\ell,d} \equiv \beta_{q,d}.\]
\end{corollary}
We now show that the  leverage function can be used to randomly sample the points on the unit sphere to discretize the regression problem in \cref{eq:regression-problem} and solve it approximately.

\begin{restatable}[Approximate Regression via Leverage Function Sampling]{theorem}{approxRegLevScore} \label{thm:discretizing-regression-leverage}
For any $\epsilon, \delta >0$, let $s = c \cdot  \frac{\beta_{q,d}}{\epsilon^2} \left( \log \beta_{q,d} + \delta^{-1} \right)$, for sufficiently large fixed constant $c$, and let $w_1, w_2, \ldots , w_s$ be i.i.d. uniform samples on $\SS^{d-1}$. Define the quasi-matrix $\P: \RR^{s} \to L^2(\SS^{d-1})$ as follows, for every $v \in \RR^{d}$:
\[ [\P \cdot v](\sigma) := \sum_{\ell=0}^q \frac{\alpha_{\ell,d}}{\sqrt{s \cdot |\SS^{d-1}|}} \cdot \sum_{j=1}^{s} v_j \cdot P_{d}^{\ell}\left( \langle w_j , \sigma \rangle \right) ~~~\text{ for }\sigma \in \SS^{d-1}. \]
Also let $\f \in \RR^{s}$ be a vector with $\f_j := \sqrt{\frac{|\SS^{d-1}|}{s}} \cdot f(w_j)$ for $j=1,2, \ldots, s$ and let $\P^*$ be the adjoint of $\P$. 
If $\tilde{g}$ is an optimal solution to the following least-squares problem
\[ \tilde{g} \in \argmin_{g \in L^2 \left( \SS^{d-1} \right)} \left\| \P^* g - \f \right\|_2^2, \]
then with probability at least $1-\delta$ the following holds,
\[ \left\| \Kc^{(q)}_{d} \tilde{g} - f \right\|_{\SS^{d-1}}^2 \le (1+\epsilon) \cdot \min_{g \in L^2\left( \SS^{d-1} \right)} \left\| \Kc^{(q)}_{d} g - f \right\|_{\SS^{d-1}}^2. \]
\end{restatable}

We prove this theorem in \cref{apndx-proof-leverage-score-regression}.
\cref{thm:discretizing-regression-leverage} shows that the function $\tilde{g}$ obtained from solving the discretized regression problem provides an approximate solution to \cref{eq:regression-problem}.

\subsection{Efficient Solution for the Discretized Least-Squares Problem}
In this section, we demonstrate how to apply \cref{thm:discretizing-regression-leverage} algorithmically to approximately solve the regression problem of \cref{eq:regression-problem}. Specifically, we show how to use the \emph{kernel trick} to solve the randomly discretized least squares problem efficiently.

\begin{algorithm}[h]
\caption{Efficient Spherical Harmonic Expansion} \label{alg:leverage-sample-regression}
\begin{algorithmic}[1]
\State {\bf input:} accuracy parameter $\epsilon>0$, failure probability $\delta \in (0,1)$, integer $q\ge 0$
\State Set $s = c \cdot \frac{\beta_{q,d}}{\epsilon^2} (\log \beta_{q,d} + \delta^{-1})$ for sufficiently large fixed constant $c$
\State{Sample i.i.d. random points $w_1, w_2, \ldots , w_s$ from a uniform distribution on $\SS^{d-1}$}\label{alg-line-randompoints-w}
\State{Compute $\K \in \RR^{s\times s}$ with $\K_{i,j} = \sum_{\ell=0}^q \frac{\alpha_{\ell,d}}{s} \cdot P_{d}^{\ell}\left( \langle w_i , w_j \rangle \right)$ for $i,j \in [s]$}\label{alg-line-kernel-def}
\State{Compute $\f \in \RR^{s}$ with $\f_j = \sqrt{\frac{|\SS^{d-1}|}{s}} \cdot f(w_j)$ for $j\in [s]$}\label{alg-line-sampled-function}
\State{Solve the regression by computing $\z = \K^{\dagger} \f$}\label{alg-line-opt-solution-regression}
\State{{\bf return } $y \in \Hc^{(q)}(\SS^{d-1})$ with $y(\sigma) := \sum_{\ell=0}^q \frac{\alpha_{\ell,d}}{\sqrt{s \cdot |\SS^{d-1}|}} \cdot \sum_{j=1}^s \z_j \cdot P_{d}^{\ell}\left( \langle w_j , \sigma \rangle \right)$ for $\sigma \in \SS^{d-1}$}\label{alg-line-output-function}
\end{algorithmic}
\end{algorithm}

\begin{restatable}[Efficient Spherical Harmonic Interpolation]{theorem}{efficientsphericqalharmoicinterpolationtheorem}\label{thm:efficient-regression-alg}
\cref{alg:leverage-sample-regression} returns a function $y \in \Hc^{(q)}(\SS^{d-1})$ such that, with probability at least $1-\delta$:
\[ \left\| y - f^{(q)} \right\|_{\SS^{d-1}}^2 \le \epsilon \cdot \left\| f^{(q)} - f \right\|_{\SS^{d-1}}^2, \]
where $f^{(q)} := \Kc^{(q)}_{d} f$. 
Suppose we can compute
the Gegenbauer polynomial $P_{d}^{\ell}(t)$ at every point $t \in [-1,1]$ in constant time.  \cref{alg:leverage-sample-regression} queries the function $f$ at $s = \bigo \left( \frac{\beta_{q,d}}{\epsilon^2} \left( \log \beta_{q,d} + \delta^{-1} \right) \right)$
points on the sphere $\SS^{d-1}$ and runs in $\bigo(s^2 \cdot d + s^{\omega})$
time. This algorithm evaluates $y(\sigma)$ in $\bigo(d \cdot s)$ time for any $\sigma \in \SS^{d-1}$.

\end{restatable}

For a proof of this theorem see~\cref{proofofefficientalgorithm}.


\section{Lower Bound on The Number of Required Observations} \label{sec:lower-bound}
We conclude by showing that the dimensionality of the Hilbert space $\Hc^{(q)}(\SS^{d-1})$ tightly characterizes the sample complexity of \cref{interpolation-problem}. Thus, our \cref{thm:efficient-regression-alg} is optimal up to a logarithmic factor.
The crucial fact that we use for proving the lower bound is that all the eigenvalues of the operator $\Kc^{(q)}_{d}$ are equal to one. This fact follows from the addition theorem presented in \cref{lem:additive-formula}.
By this lemma, if $\left\{ y^\ell_1, y^\ell_2, \ldots , y^\ell_{\alpha_{\ell,d}} \right\}$ is an orthonormal basis of $\Hc_{\ell}\left( \SS^{d-1} \right)$, then for any function $f \in L^2\left( \SS^{d-1} \right)$,
\begin{align}
    \left[\Kc^{(q)}_{d} f \right](\sigma) &= \sum_{\ell=0}^q \alpha_{\ell,d} \cdot \E_{w \sim \mathcal{U}(\SS^{d-1}) } \left[ P_{d}^{\ell}\left( \langle \sigma , w \rangle \right) \cdot f(w) \right]\nonumber\\
    &= \sum_{\ell=0}^q \left| \SS^{d-1} \right| \cdot \E_{w \sim \mathcal{U}(\SS^{d-1}) }\left[ \sum_{j=1}^{\alpha_{\ell,d}} y^\ell_j(\sigma) \cdot y^\ell_j(w) \cdot f(w) \right]\nonumber\\
    &= \sum_{\ell=0}^q \sum_{j=1}^{\alpha_{\ell,d}} \langle y^\ell_j, f \rangle_{\SS^{d-1}} \cdot y^\ell_j(\sigma). \label{eq:eigendecomposition-kernel-operator}
\end{align}
This shows that all (non-zero) eigenvalues of the operator $\Kc^{(q)}_{d}$ are equal to one.

\begin{theorem}[Lower Bound]\label{thm:lower-bound} Consider an error parameter $\epsilon > 0$, and any (possibly randomized) algorithm that solves \cref{interpolation-problem} with probability greater than $1/10$ for any input function $f$ and makes at most $r$ (possibly adaptive) queries on any input. Then $r \ge \beta_{q,d}$.
\end{theorem}

We prove this lower bound by describing a distribution on the input functions $f$ on which any deterministic algorithm that takes $r < \beta_{q,d}$ samples on any input fails with probability greater than $9/10$. The theorem then follows by Yao’s principle.

\paragraph{Hard Input Distribution.} For any integer $\ell \le q$, consider an orthonormal basis of $\Hc_{\ell}\left( \SS^{d-1} \right)$ and denote it by $\left\{ y^\ell_1, y^\ell_2, \ldots, y^\ell_{\alpha_{\ell,d}} \right\}$. 
Let $\Y_{\ell} :\RR^{\alpha_{\ell,d}} \to \Hc_{\ell}\left( \SS^{d-1} \right)$ be the quasi-matrix with $y^\ell_j$ as its $j^{th}$ column, i.e., $[\Y_{\ell} \cdot u](\sigma) := \sum_{j=1}^{\alpha_{\ell,d}} u_j \cdot y^\ell_j(\sigma)$ for any $u \in \RR^{\alpha_{\ell,d}}$ and $\sigma \in \SS^{d-1}$.
Let vectors $v^{(0)} \in \RR^{\alpha_{0,d}}, v^{(1)} \in \RR^{\alpha_{1,d}}, \ldots, v^{(q)} \in \RR^{\alpha_{q,d}}$ be independent random vectors with each entry distributed independently as a Gaussian: $v^{(\ell)}_j \sim \mathcal{N}(0,1)$.
The random input is defined to be $f := \sum_{\ell=0}^q \Y_{\ell} \cdot v^{(\ell)}$. In other words, $f = \sum_{\ell=0}^q \Y_{\ell} \cdot v^{(\ell)}$ is a random linear combination of the eigenfunctions of $\Kc^{(q)}_{d}$.

We prove that accurate reconstruction of function $f$ drawn from the above-mentioned hard input distribution yields an accurate reconstruction of the random vectors $v^{(0)}, v^{(1)}, \ldots, v^{(q)}$. Since each $v^{(\ell)}$ is $\alpha_{\ell,d}$-dimensional, this reconstruction requires $\Omega(\sum_{\ell=0}^q \alpha_{\ell,d}) = \Omega( \beta_{q,d})$ samples, giving us a lower bound for accurately reconstructing $f$.

\begin{restatable}{claim}{claimlowerboundhardinputneedszeroregressionerror}\label{problem1-guarantee-lowerbound}
Given the random input $f = \sum_{\ell=0}^q \Y_{\ell} \cdot v^{(\ell)}$ generated as described above, to solve \cref{interpolation-problem}, an algorithm must return a function $\tilde{f}^{(q)} \in \Hc^{(q)}\left( \SS^{d-1} \right)$ such that $\| \tilde{f}^{(q)} - f \|_{\SS^{d-1}}^2 =0$.
\end{restatable}
We prove this claim in \cref{appndx-lower-bound}.
Now we show that finding an $\tilde{f}^{(q)}$ satisfying the condition of \cref{interpolation-problem} is at least as hard as accurately finding all vectors $v^{(0)}, v^{(1)}, \ldots, v^{(q)}$.

\begin{restatable}{lemma}{lemlowerboundhardinstancecoeffsmustberecoverable}\label{lem:problem1-ashardas-finding-v}
If a deterministic algorithm solves \cref{interpolation-problem} with probability at least $1/10$ over our random input distribution $f = \sum_{\ell=0}^q \Y_{\ell} \cdot v^{(\ell)}$, then with probability at least $1/10$, the output of the algorithm $\tilde{f}^{(q)}$ satisfies $\Y_{\ell}^* \tilde{f}^{(q)} = v^{(\ell)}$ for all integers $\ell\le q$.
\end{restatable}

Finally, we complete the proof of \cref{thm:lower-bound} by arguing that if $\tilde{f}^{(q)}$ is formed using less than $\beta_{q,d}$ queries from $f$, then $\sum_{\ell=0}^q \left\| \Y_{\ell}^* \tilde{f}^{(q)} - v^{(\ell)} \right\|_2^2 > 0$ with good probability. Thus the bound of \cref{lem:problem1-ashardas-finding-v} cannot hold and so $\tilde{f}^{(q)}$ cannot be a solution to \cref{interpolation-problem} with good probability.
Assume for the sake of contradiction that there is a deterministic algorithm which solves \cref{interpolation-problem} with probability at least $1/10$ over the random input $f = \sum_{\ell=0}^q \Y_{\ell} \cdot v^{(\ell)}$ that makes $r = \beta_{q,d} - 1$ queries on any input (we can always modify an algorithm that makes fewer queries on some inputs to make exactly $\beta_{\ell,d} - 1$ queries and return the same output).
For every $\sigma \in \SS^{d-1}$ and integer $\ell \le q$ let the vector $u_{\sigma}^{\ell} \in \RR^{\alpha_{\ell,d}}$ be defined as $u_{\sigma}^\ell := \left[ y^\ell_1(\sigma), y^\ell_2(\sigma), \ldots, y^\ell_{\alpha_{\ell,d}}(\sigma) \right]$. Also define $\u_\sigma \in \RR^{\beta_{q,d}}$ as $\u_\sigma := \left[ u^0_\sigma, u^1_\sigma, \ldots, u^q_\sigma \right]$.
Furthermore, define $\v \in \RR^{\beta_{q,d}}$ as $\v := \left( v^{(0)}, v^{(1)}, \ldots, v^{(q)} \right)$. Additionally, define the quasi-matrix $\Y := \left[ \Y_0, \ldots, \Y_q \right]$.

Using the above notations and the definition of the hard input instance $f$, each query to $f$ is in fact a query to the random vector $\v$ in the form of $f(\sigma) = \langle \u_\sigma, \v \rangle$. Now consider a deterministic function $Q$, that is given input $\V \in \RR^{i \times \beta_{q,d}}$ (for any positive integer $i$) and outputs $Q(\V) \in \RR^{\beta_{q,d} \times \beta_{q,d}}$ such that $Q(\V)$ has orthonormal rows with the first $i$ rows spanning the $i$ rows of $\V$. 
If $\sigma_1, \sigma_2, \ldots, \sigma_r \in \SS^{d-1}$ denote the points where our algorithm queries the input $f$, for any integer $i \in [r]$, let: 
\[\Q^i := Q\left( [ \u_{\sigma_1}, \u_{\sigma_2}, \ldots, \u_{\sigma_i} ]^\top \right).\]
That is $\Q^i$ is an orthonormal matrix whose first $i$ rows span the first $i$ queries of the algorithm. Note that since our algorithm is deterministic, $\Q^i$ is a deterministic function of the random input $\v$. We have the following claim from \cite{avron2019universal}:
\begin{claim}[Claim~23 of \cite{avron2019universal}]\label{claim:distribution-queries-deterministice-alg}
Conditioned on the queries $f(\sigma_1), f(\sigma_2), \ldots, f(\sigma_r)$ for $r < \beta_{q,d}$, the variable $[\Q^r\cdot \v]({\beta_{q,d}})$ is distributed as $\mathcal{N}(0,1)$.
\end{claim}
Now using \cref{claim:distribution-queries-deterministice-alg} we can write,
\begin{align*}
    \Pr_{\v}\left[ \sum_{\ell=0}^q \left\| v^{(\ell)} - \Y_{\ell}^* \tilde{f}^{(q)} \right\|_2^2 = 0 \right] &= \Pr_{\v}\left[ \Q^r\cdot \v = \Q^r \Y^* \tilde{f}^{(q)} \right]~~~~~~~~~~~~~~~~~\text{(Since $\Q^r$
is orthonormal)}\\
    &\le \Pr_{\v}\left[ \left[\Q^r \v \right]({\beta_{q,d}}) = \left[ \Q^r \Y^* \tilde{f}^{(q)} \right]({\beta_{q,d}}) \right]\\
    &= \E \left[ \Pr_{\v}\left[ \left.  \left[\Q^r \v \right]({\beta_{q,d}}) = \left[ \Q^r \Y^* \tilde{f}^{(q)} \right]({\beta_{q,d}}) \right|  f(\sigma_1), \ldots, f(\sigma_r) \right] \right],
\end{align*}
where the expectation in the last line is taken over the randomness of $f(\sigma_1), \ldots, f(\sigma_r)$.
Now note that conditioned on $f(\sigma_1), \ldots, f(\sigma_r)$, the quantity $\left[ \Q^r \Y^* \tilde{f}^{(q)} \right]({\beta_{q,d}})$ is a fixed vale because the algorithm determines $\tilde{f}^{(q)}$ given the knowledge of the queries $f(\sigma_1), \ldots, f(\sigma_r)$. Furthermore, by \cref{claim:distribution-queries-deterministice-alg}, $[\Q^r\cdot \v]({\beta_{q,d}})$ is a random variable distributed as $\mathcal{N}(0,1)$, conditioned on $f(\sigma_1), \ldots, f(\sigma_r)$. This implies that,
\[ \Pr\left[ \left.  \left[\Q^r\cdot \v\right]({\beta_{q,d}}) = \left[ \Q^r \Y^* \tilde{f}^{(q)} \right]({\beta_{q,d}}) \right|  f(\sigma_1), \ldots, f(\sigma_r) \right] = 0. \]
Thus,
\[ \Pr\left[ \sum_{\ell=0}^q \left\| v^{(\ell)} - \Y_{\ell}^* \tilde{f}^{(q)} \right\|_2^2 = 0 \right] = \E_{f(\sigma_1), \ldots, f(\sigma_r)} [ 0 ] = 0. \]
However, we have assumed that this algorithm solves \cref{interpolation-problem} with probability at least $1/10$, and hence,
by \cref{lem:problem1-ashardas-finding-v}, $\Pr\big[ \sum_{\ell=0}^q \| v^{(\ell)} - \Y_{\ell}^* \tilde{f}^{(q)} \|_2^2 = 0 \big] \ge 1/10$. This is a contradiction, yielding \cref{thm:lower-bound}.

\begin{figure}[t]
    \centering
	\subfigure[$d=3$]{\includegraphics[width=0.48\linewidth]{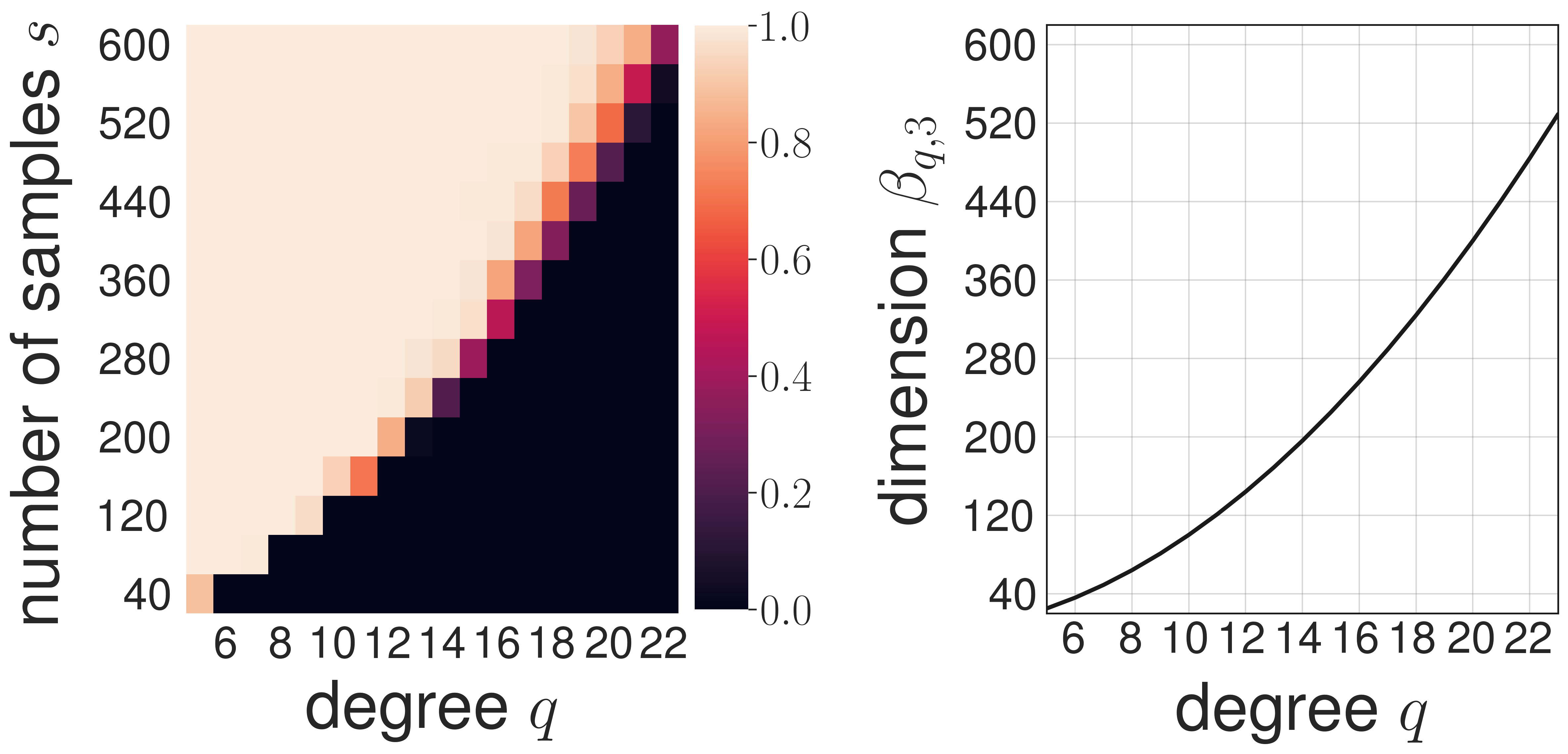}\label{fig:exp_d3}}
	\hspace{0.05in}
	\subfigure[$d=4$]{\includegraphics[width=0.48\linewidth]{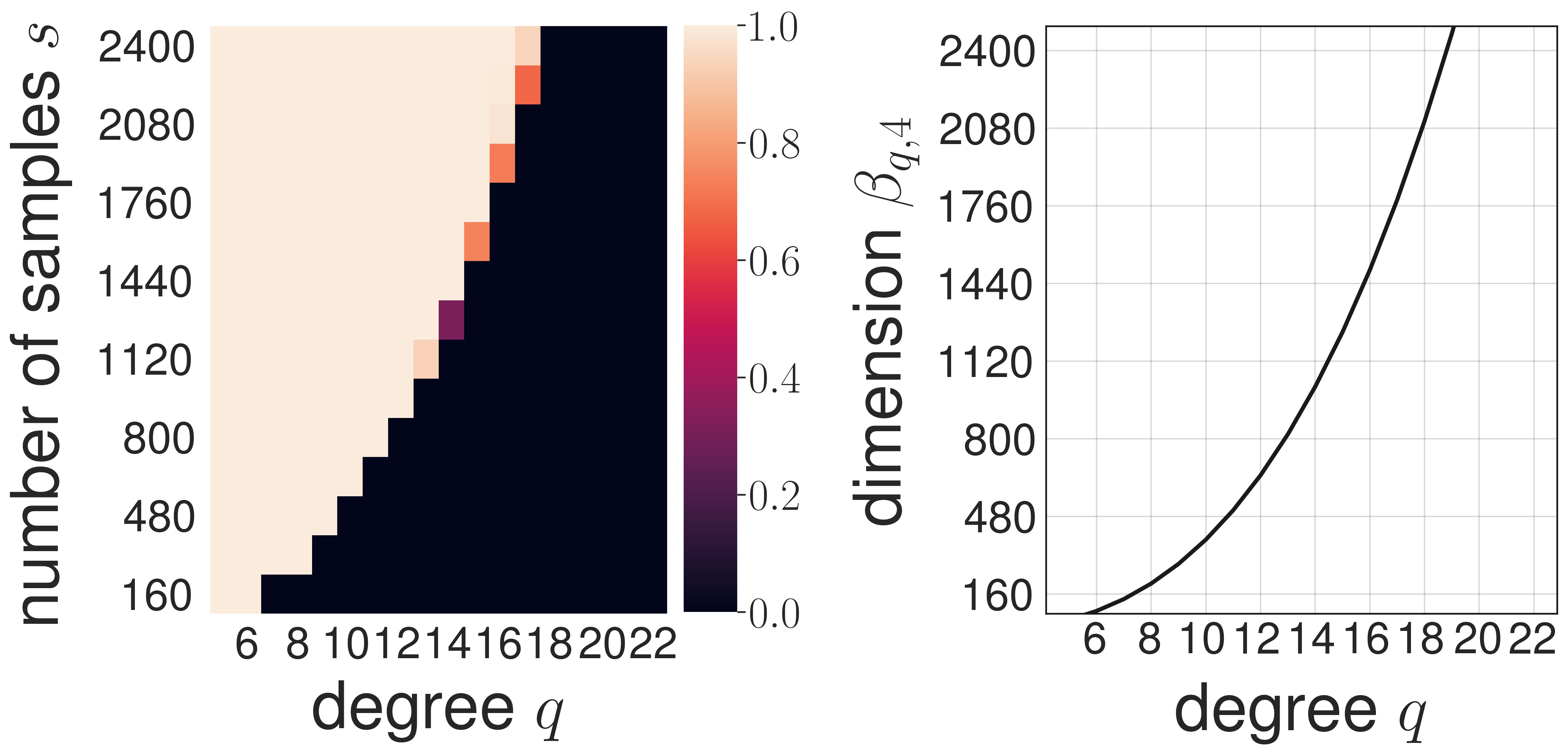}\label{fig:exp_d4}}
	\vspace{-0.1in}
	\caption{(Left) Empirical success probabilities of \cref{alg:leverage-sample-regression} varying the number of samples $s$ and the degree of spherical harmonic expansion $q$. (Right) The dimension $\beta_{q,d}$ of the Hilbert space $\Hc^{(q)}(\SS^{d-1})$ as a function of $q$ when (a) $d=3$ and (b) $d=4$, respectively.} \label{fig:exp}
\end{figure}


\section{Numerical Performance}

We conduct experiments for verifying numerical performance of our algorithm. Specifically, for a fixed $q$, we generate a random function $f(\sigma) = \sum_{\ell=0}^q c_\ell P_{d}^\ell(\inner{\sigma, v})$ where $v \sim \mathcal{U} (\SS^{d-1})$ and $c_\ell$'s are i.i.d. samples from $\mathcal{N}(0,1)$. 
Then, $f$ is recovered by running \cref{alg:leverage-sample-regression} with $s$ random evaluations of $f$ on $\SS^{d-1}$. 
Note that $\|\Kc_d^{(q)} f - f\|_{\SS^{d-1}}=0$ since $f \in \Hc^{(q)}(\SS^{d-1})$, thus, as shown in~\cref{thm:discretizing-regression-leverage}, \cref{alg:leverage-sample-regression} can recover $f$ ``exactly'' using $s=\bigo(\beta_{q,d}\log \beta_{q,d})$ evaluations, where $\beta_{q,d}$ is the dimension of the Hilbert space $\Hc^{(q)}(\SS^{d-1})$.

We predict $f$'s value on a random test set on $\SS^{d-1}$ and consider the algorithm fails if the testing error is greater than $10^{-12}$. We count the number of failures among $100$ independent random trials with different choices of $d\in\{3,4\}$, $q\in \{5,\dots,22\}$, and $s\in\{40,\dots,2400\}$. The empirical success probabilities for $d=3$ and $4$ are reported in \cref{fig:exp_d3} and \cref{fig:exp_d4}, respectively.

\cref{fig:exp} illustrates that the success probabilities of our algorithm sharply transition to $1$ as soon as the number of samples approaches $s \approx \beta_{q,d}$ for a wide range of $q$ and both $d=3,4$.
These experimental results complement our \cref{thm:discretizing-regression-leverage} along with the lower bound analysis in \cref{sec:lower-bound} and verify the empirical performance of our algorithm.

\section*{Acknowledgements}

Haim Avron was partially supported by the Israel Science Foundation (grant no. 1272/17) and by the US-Israel Binational Science Foundation (grant no. 2017698). 
Amir Zandieh was supported by the Swiss NSF grant No. P2ELP2\textunderscore 195140. Insu Han was supported by TATA DATA Analysis (grant no. 105676).

\bibliography{references.bib}
\bibliographystyle{alpha}

\clearpage
\appendix

\section{Properties of Gegenbauer Polynomials and Spherical Harmonics}\label{appndx-preliminaries}
In this section we prove the basic properties of the Gegenbauer Polynomials as well as the Spherical Harmonics and establish the connection between the two.
We start by the direct sum decomposition of the Hilbert space $L^2(\SS^{d-1})$ in terms of the spherical harmonics,

\directsumdecompose*
\begin{proof}
This is in fact a standard result. For example, see \cite{lang2012real} for a proof.

\end{proof}

Now we show that the Gegenbauer polynomials and spherical harmonics are related through the so called addition theorem,

\additiveformula*
\begin{proof}
The result can be proven analytically, using the properties of the Poisson kernel in the unit ball. This is classic and the proof can be found in \citep[Theorem 2.9]{atkinson2012spherical}.

\end{proof}

Next we show that the Gegenbauer kernels can project any function into the space of their corresponding spherical harmonics,

\projectionltofunctionsintohilberpspacesphericharmonic*
\begin{proof}
This is a classic textbook result, see \cite{morimoto1998analytic}.

\end{proof}

Now we prove that the Gegenbauer kernels satisfy the reproducing property for the Hilbert space $\Hc_\ell(\SS^{d-1})$.

\reproducingpropertygegenbauerkernels*
\begin{proof}
This result follows directly from the Funk–Hecke formula (See~\cite{atkinson2012spherical}).
However, we provide another proof here. First note that for every $x \in \SS^{d-1}$ the function $P_{d}^{\ell}\left( \langle x , \cdot \rangle \right) \in \Hc_{\ell}\left( \SS^{d-1} \right)$. Therefore the first claim follow by applying \cref{lem:decompose-Gegenbauer} on function $f(\sigma) = P_{d}^{\ell}\left( \langle x , \sigma \rangle \right)$ which also satisfies $f_\ell = f$.
On the other hand, $P_{d}^{\ell'}\left( \langle y , \cdot \rangle \right) \in \Hc_{\ell'}\left( \SS^{d-1} \right)$ for every $y \in \SS^{d-1}$. Thus, for $\ell' \neq \ell$, using the fact that spherical harmonics are orthogonal spaces of functions, $P_{d}^{\ell'}\left( \langle y , \cdot \rangle \right) \perp \Hc_{\ell}\left( \SS^{d-1} \right)$, which gives the second claim. 

\end{proof}

Next we prove that the kernel operator defined in \cref{eq:kernel-operator-def} is in fact a projection operator,

\kerneloperatorisprojectionoperator*
\begin{proof}
For every $f \in L^2\left( \SS^{d-1} \right)$ and every $\sigma \in \SS^{d-1}$, using \cref{eq:kernel-operator-def} we have,
\begin{align*}
    \left[ \left(\Kc^{(q)}_{d}\right)^2 f\right](\sigma) &= \sum_{\ell' =0}^q \frac{\alpha_{\ell',d}}{|\SS^{d-1}|} \left< \Kc^{(q)}_{d} f, P_{d}^{\ell'}\left( \langle \sigma , \cdot \rangle \right) \right>_{\SS^{d-1}}\\ 
    &= \sum_{\ell=0}^q\sum_{\ell' = 0}^q \alpha_{\ell,d} \alpha_{\ell',d} \cdot \E_{w \sim \mathcal{U}(\SS^{d-1})} \left[ P_{d}^{\ell'}\left( \langle \sigma , w \rangle \right) \cdot \E_{\tau \sim \mathcal{U}(\SS^{d-1})}\left[ P_{d}^{\ell}\left( \langle \tau , w \rangle \right) f(\tau) \right] \right]\\
    &= \sum_{\ell=0}^q\sum_{\ell' = 0}^q \alpha_{\ell,d} \alpha_{\ell',d} \cdot \E_{\tau \sim \mathcal{U}(\SS^{d-1})}\left[ f(\tau) \cdot \E_{w \sim \mathcal{U}(\SS^{d-1})} \left[
    P_{d}^{\ell'}\left( \langle \sigma , w \rangle \right)  P_{d}^{\ell}\left( \langle \tau , w \rangle \right) \right] \right]\\
    &= \sum_{\ell=0}^q \alpha_{\ell,d} \cdot \E_{\tau \sim \mathcal{U}(\SS^{d-1})}\left[ f(\tau) \cdot
    P_{d}^{\ell}\left( \langle \sigma , \tau \rangle \right)  \right]\\ 
    &= \left[ \Kc^{(q)}_{d} f\right](\sigma),
\end{align*}
where the fourth line above follows from \cref{claim-gegen-kernel-properties}.
This proves the claim.

\end{proof}

\section{Reducing the Interpolation Problem to a Least-Squares Regression}\label{appendix-reduce-interpolation-approximate-LS}
In this section we show that our spherical harmonic interpolation problem, i.e., \cref{interpolation-problem}, can be solved by approximately solving a least-squares problem as claimed in \cref{claim-interpolation-via-least-squares}.
We start by showing that for any function $f \in L^2(\SS^{d-1})$, $\Kc^{(q)}_{d} f$ gives its low-degree component. More precisely, let $f = \sum_{\ell=0}^\infty f_\ell$ the the decomposition of $f$ over the Hilbert sum $\bigoplus_{\ell=0}^{\infty} \Hc_{\ell}\left( \SS^{d-1} \right)$ as per \cref{lem:spherical-harmonic-direct-sum-decompose}. Now if we let $\Kc^{(q)}_{d}$ be the kernel operator from \cref{eq:kernel-operator-def} and if $\left\{ y^\ell_1, y^\ell_2, \ldots, y^\ell_{\alpha_{\ell,d}} \right\}$ is an orthonormal basis for $\Hc_{\ell}\left( \SS^{d-1} \right)$, then by \cref{lem:additive-formula} we have,
\begin{align*}
    \left[ \Kc^{(q)}_{d} f \right](\sigma) &= \sum_{\ell =0}^q \alpha_{\ell,d} \cdot \E_{w \sim \mathcal{U} (\SS^{d-1}) } \left[ f(w) P_{d}^{\ell}\left( \langle \sigma , w \rangle \right) \right] \\
    &= \sum_{\ell =0}^q \left| \SS^{d-1} \right| \cdot \E_{w \sim \mathcal{U} (\SS^{d-1}) } \left[ f(w) \cdot \sum_{j=1}^{\alpha_{\ell,d}} y^\ell_j(\sigma) \cdot y^\ell_j(w) \right]\\
    &= \sum_{\ell =0}^q \sum_{j=1}^{\alpha_{\ell,d}} y^\ell_j(\sigma) \cdot \left| \SS^{d-1} \right| \cdot \E_{w \sim \mathcal{U} (\SS^{d-1}) } \left[ f(w) \cdot y^\ell_j(w) \right]\\
    &= \sum_{\ell =0}^q \sum_{j=1}^{\alpha_{\ell,d}}  \langle f , y^\ell_j(w) \rangle_{\SS^{d-1}} \cdot y^\ell_j(\sigma) \\
    &= \sum_{\ell =0}^q f_\ell(\sigma) = f^{(q)}(\sigma),
\end{align*}
where the the second line above follows from \cref{lem:additive-formula}, the fourth line follows from \cref{eq:inner-prod-def}, and the last line follows from \cref{lem:spherical-harmonic-direct-sum-decompose}. This proves that the low-degree component $f^{(q)} = \Kc^{(q)}_{d} f$.

\approximateregressionsolvesinterpolation*
\begin{proof}
First, note that $g^* = f$ is an optimal solution to the least-squares problem in \cref{eq:regression-problem}. Thus we have,
\[ \min_{g \in L^2\left( \SS^{d-1} \right)} \left\| \Kc^{(q)}_{d} g - f \right\|_{\SS^{d-1}}^2 = \left\| \Kc^{(q)}_{d} f - f \right\|_{\SS^{d-1}}^2 = \left\| f^{(q)} - f \right\|_{\SS^{d-1}}^2. \]
Next, we can write,
\begin{align*}
    \left\| \Kc^{(q)}_{d} \tilde{g} - f \right\|_{\SS^{d-1}}^2 &= \left\| \Kc^{(q)}_{d} \tilde{g} - \Kc^{(q)}_{d}f + \left( \Kc^{(q)}_{d}f - f \right) \right\|_{\SS^{d-1}}^2\\
    &= \left\| \Kc^{(q)}_{d} (\tilde{g} - f) + \left( \Kc^{(q)}_{d}f - f \right) \right\|_{\SS^{d-1}}^2\\
    &= \left\| \Kc^{(q)}_{d} (\tilde{g} - f)\right\|_{\SS^{d-1}}^2 + \left\| \Kc^{(q)}_{d}f - f \right\|_{\SS^{d-1}}^2\\
    &= \left\| \Kc^{(q)}_{d} \tilde{g} - f^{(q)} \right\|_{\SS^{d-1}}^2 + \left\| f^{(q)} - f  \right\|_{\SS^{d-1}}^2,
\end{align*}
where the third line follows from the Pythagorean theorem because $\Kc^{(q)}_{d} (\tilde{g} - f) \in \Hc^{(q)}\left( \SS^{d-1} \right)$ while $ \Kc^{(q)}_{d}f - f = -\sum_{\ell>q} f_\ell$, thus $\left( \Kc^{(q)}_{d}f - f \right) \perp  \Hc^{(q)}\left( \SS^{d-1} \right) $. Combining the two equalities above with inequality $\left\| \Kc^{(q)}_{d} \tilde{g} - f \right\|_{\SS^{d-1}}^2 \le C \cdot \min_{g \in L^2\left( \SS^{d-1} \right)} \left\| \Kc^{(q)}_{d} g - f \right\|_{\SS^{d-1}}^2$ was given in the statement of the claim proves the \cref{claim-interpolation-via-least-squares}.

\end{proof}

\section{Approximate Regression via Leverage Score Sampling}\label{apndx-proof-leverage-score-regression}
In this section we ultimately prove our main result of \cref{thm:discretizing-regression-leverage}. We start by proving useful properties of the leverage function given in \cref{def:leverage-function}.
First, we show the fact that the ridge leverage function can be characterized in terms of a least-squares minimization problem, which is crucial for computing the leverage scores distribution. This fact was previously exploited in \cite{avron2017random} and \cite{avron2019universal} in the context of Fourier operators.

\mincharacterleveragefunction*
We remark that this lemma is in fact an adaptation and generalization of Theorem~5 of \cite{avron2019universal}. We prove this lemma here for the sake of completeness.
\begin{proof}
First we show that the right hand side of \cref{eq:lev-score-min-char} is never smaller than the leverage function in \cref{def:leverage-function}. Let $g_w^* \in L^2(\SS^{d-1})$ be the optimal solution of \cref{eq:lev-score-min-char} for any $w \in \SS^{d-1}$. Note that the optimal solution satisfies $\Kc^{(q)}_{d} g_w^* = \phi_w$. Thus, for any function $f \in L^2(\SS^{d-1})$, using \cref{eq:kernel-operator-def}, we can write
\begin{align*}
    \left| \left[ \Kc^{(q)}_{d} f\right](w) \right|^2 &= \left| \sum_{\ell=0}^q \alpha_{\ell,d} \cdot \E_{\sigma \sim \mathcal{U}(\SS^{d-1})} \left[ P_{d}^{\ell}\left( \langle \sigma , w \rangle \right) \cdot f(\sigma) \right] \right|^2\\
    &= \left| \langle \phi_w, f \rangle_{\SS^{d-1}} \right|^2 = \left| \left< \Kc^{(q)}_{d} g_w^*, f \right>_{\SS^{d-1}} \right|^2\\
    &= \left| \left< g_w^*, \Kc^{(q)}_{d} f \right>_{\SS^{d-1}} \right|^2 ~~~~~~~~~~~~~~~~~~~~~~~~~~~~~~~~~~~~~~~~~~~~\text{(because $\Kc^{(q)}_{d}$ is self-adjoint)}\\
    &\le \|g_w^*\|_{\SS^{d-1}}^2 \cdot \left\| \Kc^{(q)}_{d} f \right\|_{\SS^{d-1}}^2 ~~~~~~~~~~~~~~~~~~~~~~~~~~~~~~\text{(by Cauchy–Schwarz inequality)}
\end{align*}
Therefore, for any $f \in L^2(\SS^{d-1})$ with $\left\| \Kc^{(q)}_{d}f \right\|_{\SS^{d-1}}>0$, we have
\begin{equation}\label{eq:max-char-smaller-min-char}
\frac{\left| \left[ \Kc^{(q)}_{d} f\right](w) \right|^2}{\left\| \Kc^{(q)}_{d} f \right\|_{\SS^{d-1}}^2} \le  \|g_w^*\|_{\SS^{d-1}}^2.     
\end{equation}

We conclude the proof by showing that the maximum value is attained. First, we show that the optimal solution $g_w^*$ of \cref{eq:lev-score-min-char} satisfies the property that $\Kc^{(q)}_{d} g_w^* = g_w^*$. Suppose for the sake of contradiction that $\Kc^{(q)}_{d} g_w^* \neq g_w^*$. In this case, \cref{claim:k-square-eqaul-k} implies that,
\[ \Kc^{(q)}_{d} \left(\Kc^{(q)}_{d} g_w^* - g_w^*\right) = \left(\Kc^{(q)}_{d}\right)^2 g_w^* - \Kc^{(q)}_{d} g_w^* = \Kc^{(q)}_{d} g_w^* - \Kc^{(q)}_{d} g_w^* = 0. \]
Thus, the function $g = \Kc^{(q)}_{d} g_w^*$ satisfies the constraint of the minimization problem in \cref{eq:lev-score-min-char}. Now, using the above and the fact that $\Kc^{(q)}_{d}$ is self-adjoint we can write,
\[ \left< \Kc^{(q)}_{d} g_w^*,  \Kc^{(q)}_{d} g_w^* - g_w^* \right>_{\SS^{d-1}} = \left< g_w^*,  \Kc^{(q)}_{d} \left(\Kc^{(q)}_{d} g_w^* - g_w^* \right) \right>_{\SS^{d-1}} = 0. \]
This shows that $\Kc^{(q)}_{d} g_w^* \perp \left(\Kc^{(q)}_{d} g_w^* - g_w^* \right)$, hence by Pythagorean theorem we have,
\[ \|g_w^*\|_{\SS^{d-1}}^2 = \left\|\Kc^{(q)}_{d} g_w^*\right\|_{\SS^{d-1}}^2 + \left\|\Kc^{(q)}_{d} g_w^* - g_w^*\right\|_{\SS^{d-1}}^2 > \left\|\Kc^{(q)}_{d} g_w^*\right\|_{\SS^{d-1}}^2 = \left\|g\right\|_{\SS^{d-1}}^2, \]
which is in contrast with the assumption that $g_w^*$ is the optimal solution of \cref{eq:lev-score-min-char}. Therefore, our claim that $\Kc^{(\ell)}_{d} g_w^* = g_w^*$ holds.

Now, we show that for $f = g_w^*$, the maximum value in inequality \cref{eq:max-char-smaller-min-char} is attained. For any $w \in \SS^{d-1}$ we have the following
\[ \left[ \Kc^{(q)}_{d} f\right](w) =  \left< \Kc^{(q)}_{d} g_w^*, f \right>_{\SS^{d-1}} = \left< g_w^*, g_w^* \right>_{\SS^{d-1}} = \|g_w^*\|_{\SS^{d-1}}^2. \]
On the other hand we have $\left\| \Kc^{(q)}_{d} f\right\|_{\SS^{d-1}}^2 = \|g_w^*\|_{\SS^{d-1}}^2$.
Thus, $ \left\| \Kc^{(q)}_{d} f\right\|_{\SS^{d-1}}^{-2} \cdot \left| \left[ \Kc^{(q)}_{d} f\right](w) \right|^2 = \|g_w^*\|_{\SS^{d-1}}^2 $ which implies that $\tau_q(w) = \|g_w^*\|_{\SS^{d-1}}^2$ and thus proves the lemma.

\end{proof}

To prove our \cref{thm:discretizing-regression-leverage}, we need to use some results about concentration of random operators. In particular we use Lemma~37 from \cite{avron2019universal}, which is restated bellow,

\begin{lemma}[Lemma 37 of \cite{avron2019universal}]\label{lem:matrix-chernoff-operators}
Suppose that $\mathcal{H}$ is a separable Hilbert space, and that $\mathcal{B}$ is a fixed self-adjoint Hilbert-Schmidt operator on $\mathcal{H}$. Let $\Rc$ be a self-adjoint Hilbert-Schmidt random operator that satisfies
\[ \E [\Rc] = \mathcal{B} , \text{ and } \| \Rc \|_{op} \le L \]
Let $\mathcal{M}$ be another self-adjoint trace-class operator such that $\E [\Rc^2] \preceq \mathcal{M}$. Form the operator sampling estimator
\[ \Bar{\Rc}_n := \frac{1}{n} \sum_{k=1}^n \Rc_k \]
where each $\Rc_k$ is an independent copy of $\Rc$. Then, for any $t > \sqrt{ \| \mathcal{M} \|_{op} / n} + 2L/3n$,
\[ \Pr[ \| \Bar{\Rc}_n - \mathcal{B} \|_{op} > t ] \le \frac{8 \cdot \trace(\mathcal{M})}{\| \mathcal{M} \|_{op}} \cdot \exp\left( \frac{-n t^2 / 2}{ \| \mathcal{M} \|_{op} + 2Lt/3 } \right). \]
\end{lemma}

Our approach is to apply \cref{lem:matrix-chernoff-operators} to show that the operator $\Kc_d^{(q)}$ can be well approximated by $\P \P^*$, where the quasi-matrix $\P$ is defined in \cref{thm:discretizing-regression-leverage}. 
In order to prove this formally, we need to define the notion of positive definiteness for self-adjoint operators.
We call self-adjoint $A : L^2(\SS^{d-1}) \to L^2(\SS^{d-1})$ positive semidefinite (or simply positive) and write $A \succeq 0$ if $\langle x, Ax \rangle_{\SS^{d-1}} \ge 0$ for all $x \in L^2(\SS^{d-1})$. 
The notation for $A \preceq B$ and $A \succeq B$ follow in the standard way.
Now with the notations in place we can prove the following lemma,

\begin{lemma}[Approximating $\Kc_d^{(q)}$ via Leverage Score Sampling]\label{lem:K-operator-approximation-leverage-score-sampling}
For any $\delta >0$ and $\epsilon \in (0,1/2)$, let $s = \frac{8\beta_{q,d}}{3\epsilon^2} \log \frac{8\beta_{q,d}}{\delta}$, for sufficiently large fixed constant $c$, and let $w_1, w_2, \ldots , w_s$ be i.i.d. uniform samples from $\SS^{d-1}$. Let $\P: \RR^{s} \to L^2(\SS^{d-1})$ be the quasi-matrix defined as follows, for every $v \in \RR^{d}$ and $\sigma \in \SS^{d-1}$:
\[ [\P \cdot v](\sigma) := \sum_{\ell=0}^q \frac{\alpha_{\ell,d}}{\sqrt{s \cdot |\SS^{d-1}|}} \cdot \sum_{j=1}^{s} v_j \cdot P_{d}^{\ell}\left( \langle w_j , \sigma \rangle \right). \]
Also let $\P^*$ be the adjoint of $\P$. Then with probability at least $1 - \delta$,
\[ (1 - \varepsilon) \cdot \Kc_d^{(q)} \preceq \P \P^* \preceq (1 + \varepsilon) \cdot \Kc_d^{(q)}. \]
\end{lemma}

\begin{proof}
The proof is by invoking \cref{lem:matrix-chernoff-operators}. The reason we can invoke this lemma is because $\Kc_d^{(q)}$ is a self adjoint trace-class orthonormal projection operator, by \cref{claim:k-square-eqaul-k} and \cref{eq:operator-K-trace-class}, thus this operator is Hilbert-Schmidt. Furthermore, the Hilbert space that we care about is $L^2(\SS^{d-1})$ which is a separable space. 

Now notice that if we define the function $\phi_w \in L^2(\SS^{d-1})$ by $\phi_w(\sigma):= \sum_{\ell=0}^q \frac{\alpha_{\ell,d}}{|\SS^{d-1}|}  P_{d}^{\ell}\left( \langle \sigma , w \rangle \right)$ for $\sigma , w \in \SS^{d-1}$, then for any $v \in \RR^s$:
\[ \P \cdot v \equiv \sqrt{\frac{|\SS^{d-1}|}{s}} \cdot \sum_{j=1}^s  v_j \cdot \phi_{w_j} \]
Furthermore, given functions $f, g \in L^2(\SS^{d-1})$ we define the operator $(f \otimes g) : L^2(\SS^{d-1}) \to L^2(\SS^{d-1})$ by \[ (f \otimes g) h := \langle g , h \rangle_{\SS^{d-1}} \cdot f ~~~~~\text{ for any } h \in L^2(\SS^{d-1}). \]
Therefore, using this notation, if we let 
\[ \Rc_j := |\SS^{d-1}| \cdot (\phi_{w_j} \otimes \phi_{w_j}), \]
then we understand that,
\[ \P \P^* \equiv \frac{1}{s} \cdot \sum_{j=1}^s \Rc_j. \]
Note that $\Rc_j$ is a rank-one self adjoint operator, thus it is also Hilbert-Schmidt. Since samples $w_1,w_2, \ldots, w_s$ are drawn independently at random, $\Rc_1, \Rc_2, \ldots, \Rc_s$ are i.i.d. random operators with expectation,
\[ \E [\Rc_j] = \left| \SS^{d-1} \right| \cdot \E_{w_j \sim \mathcal{U}(\SS^{d-1})} [\phi_{w_j} \otimes \phi_{w_j}] = \Kc_d^{(q)}. \]
The reason for the second equality above is that for any function $f \in L^2(\SS^{d-1})$ and any $\sigma \in \SS^{d-1}$:
\begin{align*}
\left| \SS^{d-1} \right| \left[ \E [\phi_{w_j} \otimes \phi_{w_j}] f \right](\sigma) &= \left| \SS^{d-1} \right| \cdot \E_{w_j} \left[ [( \phi_{w_j} \otimes \phi_{w_j} ) \cdot f ](\sigma) \right]\\
&= \left| \SS^{d-1} \right| \cdot \E_{w_j} \left[ \langle \phi_{w_j}, f \rangle_{\SS^{d-1}} \cdot \phi_{w_j}(\sigma) \right]\\
&= \sum_{\ell=0}^q \sum_{\ell'=0}^q \alpha_{\ell,d} \alpha_{\ell',d} \cdot  \E_{w } \left[ \E_{\tau} \left[ P_{d}^{\ell}\left( \langle \sigma , w \rangle \right)  P_{d}^{\ell'}\left( \langle \tau , w \rangle \right)  f(\tau) \right] \right] \\
&= \sum_{\ell=0}^q \alpha_{\ell,d} \cdot \E_{\tau \sim \mathcal{U}(\SS^{d-1})} \left[ P_{d}^{\ell}\left( \langle \sigma , \tau \rangle \right) f(\tau) \right]\\
&= \left[ \Kc_d^{(q)} f \right](\sigma),
\end{align*}
where the fourth line above follows from \cref{claim-gegen-kernel-properties}.
Next, we bound the operator norm of $\Rc_j$. This random operator only takes values that are both positive semi-definite and rank one, so the operator norm of $\Rc_j$ is equal to the following
\begin{align*}
    \| \Rc_j \|_{op} &= \left\| \left| \SS^{d-1} \right| \cdot (\phi_{w_j} \otimes \phi_{w_j}) \right\|_{op}\\
    &= \left| \SS^{d-1} \right| \cdot \left\| \phi_{w_j} \right\|_{\SS^{d-1}}^2\\
    &= \beta_{q,d},
\end{align*}
where the last line follows from \cref{claim-gegen-kernel-properties} and definition of $\phi_{w_j}$ as well as the fact that $\beta_{q,d} = \sum_{\ell=0}^q \alpha_{\ell,d}$.
The final ingredient for applying \cref{lem:matrix-chernoff-operators} is to bound $\Rc_j^2$. We have,
\begin{align*}
    \Rc_j^2 &= \left| \SS^{d-1} \right|^2 \cdot (\phi_{w_j} \otimes \phi_{w_j}) \cdot (\phi_{w_j} \otimes \phi_{w_j})\\
    &= \left| \SS^{d-1} \right|^2 \cdot \left\| \phi_{w_j} \right\|_{\SS^{d-1}}^2 \cdot (\phi_{w_j} \otimes \phi_{w_j})\\
    &= \beta_{q,d} \cdot \left| \SS^{d-1} \right| \cdot  (\phi_{w_j} \otimes \phi_{w_j})\\
    &= \beta_{q,d} \cdot \Rc_j.
\end{align*}
Therefore,
\[ \E[\Rc_j^2] = \beta_{q,d} \cdot \Kc_d^{(q)} =: \mathcal{M}. \]
Now note that by \cref{eq:operator-K-trace-class}, we have $\trace(\mathcal{M}) = \beta_{q,d} \cdot \trace\left( \Kc_d^{(q)} \right) = \beta_{q,d}^2$. Also, by \cref{claim:k-square-eqaul-k}, $\Kc_d^{(q)}$ is an orthonormal projection operator, thus $\|\mathcal{M}\|_{op} = \beta_{q,d} \cdot \left\| \Kc_d^{(q)} \right\|_{op} = \beta_{q,d}$. Therefore, by \cref{lem:matrix-chernoff-operators} we have,
\begin{align*}
    \Pr\left[ \left\| \P \P^* - \Kc_d^{(q)} \right\|_{op} > \epsilon \right] &\le \frac{8 \cdot \trace(\mathcal{M})}{\| \mathcal{M} \|_{op}} \cdot \exp\left( \frac{-s \epsilon^2 / 2}{ \| \mathcal{M} \|_{op} + 2\beta_{q,d}\epsilon/3 } \right)\\
    &= \frac{8 \cdot \beta_{q,d}^2}{\beta_{q,d}} \cdot \exp\left( \frac{-s \epsilon^2 / 2}{ \beta_{q,d} + 2\beta_{q,d}\epsilon/3 } \right)\\
    &\le \delta.
\end{align*}

Now recall that $\Kc_d^{(q)}$ is an orthonormal projection matrix. We claim that the eigenspace of $\P \P^*$ is a subspace of the eigenspace of $\Kc_d^{(q)}$. To see why note that we can write,
\begin{align}
    \Kc_d^{(q)} \cdot \P \P^* &= \Kc_d^{(q)} \cdot \left( \frac{\left| \SS^{d-1} \right|}{s} \cdot \sum_{j=1}^s (\phi_{w_j} \otimes \phi_{w_j}) \right)\nonumber\\
    &= \frac{\left| \SS^{d-1} \right|}{s} \cdot \sum_{j=1}^s \Kc_d^{(q)} \cdot  (\phi_{w_j} \otimes \phi_{w_j})\nonumber\\
    &= \frac{\left| \SS^{d-1} \right|}{s} \cdot \sum_{j=1}^s \left(\left(\Kc_d^{(q)} \cdot  \phi_{w_j}\right) \otimes \phi_{w_j} \right)\nonumber\\
    &= \frac{\left| \SS^{d-1} \right|}{s} \cdot \sum_{j=1}^s \left( \phi_{w_j} \otimes \phi_{w_j} \right)\nonumber\\
    &= \P \P^*, \label{KPPstar-equal-PPstar}
\end{align}
where the fourth line above follows because for any $\sigma , w_j \in \SS^{d-1}$,
\begin{align*}
    \left[\Kc_d^{(q)} \cdot  \phi_{w_j}\right](\sigma) &=  \sum_{\ell=0}^q \alpha_{\ell,d} \cdot \E_{v \sim \mathcal{U}(\SS^{d-1})} \left[ P_{d}^{\ell}\left( \langle \sigma , v \rangle \right) \cdot \phi_{w_j}(v) \right] \\
    &= \sum_{\ell=0}^q \sum_{\ell'=0}^q \frac{\alpha_{\ell,d} \alpha_{\ell',d}}{|\SS^{d-1}|} \cdot \E_{v \sim \mathcal{U}(\SS^{d-1})} \left[ P_{d}^{\ell}\left( \langle \sigma , v \rangle \right) \cdot P_{d}^{\ell'}\left( \langle w_j , v \rangle \right) \right] \\
    &= \sum_{\ell=0}^q \frac{\alpha_{\ell,d}}{|\SS^{d-1}|} P_{d}^{\ell}\left( \langle \sigma , w_j \rangle \right)\\
    &= \phi_{w_j}(\sigma),
\end{align*}
where the third line above follows from \cref{claim-gegen-kernel-properties}. 
Therefore, now we have shown that $\Kc_d^{(q)} \cdot \P \P^* = \P \P^*$ and $\left\| \P \P^* - \Kc_d^{(q)} \right\|_{op} \le \epsilon$. Given the fact that $\Kc_d^{(q)}$ is a symmetric self-adjoint orthonormal projection and $\P\P^*$ is also symmetric and self-adjoint, this implies that,
\[ \Pr\left[ (1-\epsilon) \cdot \Kc_d^{(q)} \preceq \P\P^* \preceq (1+\epsilon) \cdot \Kc_d^{(q)} \right] \ge 1-\delta \]
which completes the proof.

\end{proof}

Now we are ready to prove \cref{thm:discretizing-regression-leverage}. We prove this theorem by showing that for all $g \in L^2(\SS^{d-1})$, leverage function sampling lets us approximate the value of the regression objective function in \cref{eq:regression-problem} when evaluated at $g$. We do this by showing that our sampling provides the so-called \emph{affine embedding guarantee}.

\approxRegLevScore*

\begin{proof}
Throughout the proof we use $f^{(q)} := \Kc^{(q)}_{d} f$ and $B^* := \left\| f - f^{(q)} \right\|_{\SS^{d-1}}^2$.
The proof is by reduction to affine embedding. Specifically, we prove that, with probability at least $1 - \delta$, simultaneously for all $g \in L^2(\SS^{d-1})$,
\begin{equation}\label{affine-embeddng}
   (1 - \epsilon/3) \cdot \left\| \Kc^{(q)}_{d} g - f \right\|_{\SS^{d-1}}^2 \le \left\| \P^* g - \f \right\|_{2}^2 + C \le (1 + \epsilon/3) \cdot \left\| \Kc^{(q)}_{d} g - f \right\|_{\SS^{d-1}}^2,
\end{equation}
where $C$ is some fixed value independent of $g$ that only depends on $\Kc_d^{(q)}$, $\P$, $\f$, and $f$. First we show that if we can prove \cref{affine-embeddng}, then the theorem immediately follows. To see why, note that for any $\tilde{g} \in \arg\min_{g \in L^2\left( \SS^{d-1} \right)} \left\| \P^* g - \f \right\|_2^2$ we can write,
\begin{align*}
    \left\| \Kc^{(q)}_{d} \tilde{g} - f \right\|_{\SS^{d-1}}^2 &\le (1-\epsilon/3)^{-1} \left( \left\| \P^* \tilde{g} - \f \right\|_{2}^2 + C \right)   ~~~~~~~~~~~~~~~~\text{(By \cref{affine-embeddng})}\\
    &= (1-\epsilon/3)^{-1} \left( \min_{g \in L^2\left( \SS^{d-1} \right)} \left\| \P^* g - \f \right\|_2^2 + C \right)\\
    &\le (1-\epsilon/3)^{-1} \left( \left\| \P^* f - \f \right\|_{2}^2 + C \right)\\
    &\le \frac{1+\epsilon/3}{1-\epsilon/3} \cdot \left\| \Kc^{(q)}_{d} f - f \right\|_{\SS^{d-1}}^2 ~~~~~~~~~~~~~~~~~~~~~~~~\text{(By \cref{affine-embeddng})}\\
    &\le (1 + \epsilon) \cdot \min_{g \in L^2(\SS^{d-1})} \left\| \Kc^{(q)}_{d} g - f \right\|_{\SS^{d-1}}^2,
\end{align*}
where the last inequality follows because $f \in \arg\min_{g \in L^2(\SS^{d-1})} \left\| \Kc^{(q)}_{d} g - f \right\|_{\SS^{d-1}}^2$.

Thus, in order to prove the theorem it suffices to prove that the affine embedding property in \cref{affine-embeddng} holds with probability at least $1 - \delta$.

\paragraph{Expression for Least-Squares Excess Cost.}
We first show that the least-squares objective function in \cref{eq:regression-problem} can be written as a function of the deviation from the optimum $g - f$.
More specifically, for any $g \in L^2(\SS^{d-1})$ we have,
\begin{align}
    \left\| \Kc^{(q)}_{d} g - f \right\|_{\SS^{d-1}}^2 &= \left\| \Kc^{(q)}_{d} g - \Kc^{(q)}_{d} f + \Kc^{(q)}_{d} f - f \right\|_{\SS^{d-1}}^2\nonumber\\
    &= \left\| \Kc^{(q)}_{d}( g - f) + \Kc^{(q)}_{d} f - f \right\|_{\SS^{d-1}}^2 \nonumber\\
    &= \left\| \Kc^{(q)}_{d}( g - f) \right\|_{\SS^{d-1}}^2 + \left\| \Kc^{(q)}_{d} f - f \right\|_{\SS^{d-1}}^2\nonumber\\
    &= \left\| \Kc^{(q)}_{d}( g - f) \right\|_{\SS^{d-1}}^2 + B^*, \label{excess-cost-original-regression}
\end{align}
where the third line above follows from the Pythagorean theorem because $\Kc^{(q)}_{d} (g - f) \in \Hc^{(q)}\left( \SS^{d-1} \right)$ while $\left( \Kc^{(q)}_{d}f - f \right) \perp  \Hc^{(q)}\left( \SS^{d-1} \right) $.

\paragraph{Bounding The Sampling Error.}
We now show that \cref{excess-cost-original-regression} holds approximately, even after sampling. This almost immediately yields the affine embedding bound of \cref{affine-embeddng}. We can write the discretized objective function value for any $g \in L^2(\SS^{d-1})$ as,
\begin{align}
\left\| \P^* g - \f \right\|_{2}^2 &= \left\| \P^* g - \P^* f + \P^* f - \f \right\|_{2}^2 \nonumber\\
&= \left\| \P^* (g - f) + \P^* f - \f \right\|_{2}^2\nonumber\\
&= \left\| \P^* (g - f) \right\|_{2}^2 + \left\| \P^* f - \f \right\|_{2}^2 + 2\langle \P^* (g - f) , \P^* f - \f \rangle. \label{eq:samples-least-squares-cost}
\end{align}
Let us focus on the last term above. First we show that $\Kc_d^{(q)} \cdot \P = \P$. For any $v \in \RR^s$:
\begin{align*}
    \left[ \Kc_d^{(q)} \cdot \P \cdot v \right](\sigma) &= \sum_{\ell=0}^q \frac{\alpha_{\ell,d}}{\sqrt{s \cdot |\SS^{d-1}|}} \cdot \sum_{j=1}^{s} v_j \cdot \left[ \Kc_d^{(q)} P_{d}^{\ell}\left( \langle w_j , \cdot \rangle \right) \right](\sigma) \\
    &= \sum_{\ell=0}^q \frac{\alpha_{\ell,d}}{\sqrt{s \cdot |\SS^{d-1}|}} \cdot \sum_{j=1}^{s} v_j \cdot P_{d}^{\ell}\left( \langle w_j , \sigma \rangle \right) \\
    &= [ \P \cdot v ](\sigma),
\end{align*}
where the second line follows from the definition of $\Kc_d^{(q)}$ in \cref{eq:kernel-operator-def} along with \cref{claim-gegen-kernel-properties}.
Now using the fact that $\Kc_d^{(q)} \cdot \P = \P$, we can rewrite the last term as,
\begin{align*}
    \langle \P^* (g - f) , \P^* f - \f \rangle &= \langle g - f , \P (\P^* f - \f) \rangle_{\SS^{d-1}} ~~~~~~~~~~~~~~~~~~~\text{($\P$ is the adjoint of $\P^*$)}\\
    &= \left< g - f , \Kc_d^{(q)} \cdot \P (\P^* f - \f) \right>_{\SS^{d-1}}\\
    &= \left< \Kc_d^{(q)} (g - f) ,  \P (\P^* f - \f) \right>_{\SS^{d-1}} ~~~~~~~~~~~~~~\text{($\Kc_d^{(q)}$ is self-adjoint)}
\end{align*}
By plugging the above into \cref{eq:samples-least-squares-cost} and applying Cauchy-Schwarz inequality we find that,
\begin{equation}\label{sampled-least-squares-cost-breakdown}
    \left\| \P^* g - \f \right\|_{2}^2 \in \left\| \P^* (g - f) \right\|_{2}^2 + \left\| \P^* f - \f \right\|_{2}^2 \pm 2 \left\| \Kc_d^{(q)} (g - f) \right\|_{\SS^{d-1}} \cdot \left\| \P (\P^* f - \f) \right\|_{\SS^{d-1}}.
\end{equation}
Now we bound $\left\| \P (\P^* f - \f) \right\|_{\SS^{d-1}}$. We show that this quantity is small with probability at least $1 - \delta/2$, in the following claim,
\begin{claim}[Approximate Operator Application]\label{claim:sampled-f-minus-fq-small}
With probability at least $1 - \delta/2$:
\[ \left\| \P (\P^* f - \f) \right\|_{\SS^{d-1}} \le \frac{\epsilon}{18} \cdot \sqrt{B^*}. \]
\end{claim}
We prove this claim later. 
Now by plugging the bound in \cref{claim:sampled-f-minus-fq-small} into \cref{sampled-least-squares-cost-breakdown} we find that,
\begin{align*}
\left\| \P^* g - \f \right\|_{2}^2 &\in \left\| \P^* (g - f) \right\|_{2}^2 + \left\| \P^* f - \f \right\|_{2}^2 \pm \frac{\epsilon}{9} \cdot \left\| \Kc_d^{(q)} (g - f) \right\|_{\SS^{d-1}} \cdot \sqrt{B^*}\\
&\in \left\| \P^* (g - f) \right\|_{2}^2 + \left\| \P^* f - \f \right\|_{2}^2 \pm \frac{\epsilon}{18} \cdot \left\| \Kc_d^{(q)} (g - f) \right\|_{\SS^{d-1}}^2 \pm \frac{\epsilon}{18} \cdot B^*,
\end{align*}
where the second line comes from the AM-GM inequality. 
Applying the operator approximation bound of \cref{lem:K-operator-approximation-leverage-score-sampling} with error parameter $\epsilon/12$ and failure probability $\delta/2$ gives that the following holds simultaneously for all $g$, with probability at least $1 - \delta$,
\[ \left\| \P^* g - \f \right\|_{2}^2 \in (1 \pm 5\epsilon/36) \cdot \left\| \Kc_d^{(q)} (g - f) \right\|_{\SS^{d-1}}^2 + \left\| \P^* f - \f \right\|_{2}^2 \pm \frac{\epsilon}{18} \cdot B^* \]
Therefore, by plugging \cref{excess-cost-original-regression} into the above inequality we find that,
\begin{align*}
    \left\| \P^* g - \f \right\|_{2}^2 &\in (1 \pm 5\epsilon/36) \cdot \left( \left\| \Kc_d^{(q)} g - f \right\|_{\SS^{d-1}}^2 - B^* \right)+ \left\| \P^* f - \f \right\|_{2}^2 \pm \frac{\epsilon}{18} \cdot B^* \\
    &\in (1 \pm 5\epsilon/36) \cdot \left\| \Kc_d^{(q)} g - f \right\|_{\SS^{d-1}}^2 - (1 \pm 7\epsilon/36) \cdot B^* + \left\| \P^* f - \f \right\|_{2}^2\\
    &\in (1 \pm \epsilon/3) \cdot \left\| \Kc_d^{(q)} g - f \right\|_{\SS^{d-1}}^2 - B^* + \left\| \P^* f - \f \right\|_{2}^2,
\end{align*}
where the last line above follows because $B^* = \left\| \Kc_d^{(q)} f - f \right\|_{\SS^{d-1}}^2 \le \left\| \Kc_d^{(q)} g - f \right\|_{\SS^{d-1}}^2$ for any $g$.
This shows that the affine embedding guarantee of \cref{affine-embeddng} holds if we let $C := - B^* + \left\| \P^* f - \f \right\|_{2}^2$ which is a quantity that only depends on $f$, $\f$, $\P^*$, and $\Kc_d^{(q)}$ and is independent of $g$.

\end{proof}

Now we prove \cref{claim:sampled-f-minus-fq-small}.

\begin{proofof}{\cref{claim:sampled-f-minus-fq-small}}
For conciseness we use $f^{(q)} := \Kc^{(q)}_{d} f$ and also define the function $\phi_w \in L^2(\SS^{d-1})$ by $\phi_w(\sigma):= \sum_{\ell=0}^q \frac{\alpha_{\ell,d}}{|\SS^{d-1}|}  P_{d}^{\ell}\left( \langle \sigma , w \rangle \right)$ for $\sigma , w \in \SS^{d-1}$. With this definition for any $v \in \RR^s$:
\[ \P \cdot v \equiv \sqrt{\frac{|\SS^{d-1}|}{s}} \cdot \sum_{j=1}^s  v_j \cdot \phi_{w_j}. \]
Furthermore, for any $f \in L^2(\SS^{d-1})$ and any $j \in [s]$,
\begin{align*}
    [\P^* f](j) &\equiv \sqrt{\frac{|\SS^{d-1}|}{s}} \cdot \langle \phi_{w_j}, f \rangle_{\SS^{d-1}}\\
    &= \sqrt{\frac{|\SS^{d-1}|}{s}} \sum_{\ell=0}^q \alpha_{\ell,d} \cdot \E_{\sigma \sim \mathcal{U}(\SS^{d-1})} \left[ P_d^\ell (\langle \sigma, w_j \rangle) \cdot f(\sigma) \right] \\
    &= \sqrt{\frac{|\SS^{d-1}|}{s}} \cdot \left[ \Kc_d^{(q)} f \right](w_j)\\
    &= \sqrt{\frac{|\SS^{d-1}|}{s}} \cdot f^{(q)}(w_j).
\end{align*}
Therefore, if we let $\y \in \RR^s$ be the vector $\y := \P^* f - \f$, we have for any $j \in [s]$
\[ \y(j) = \sqrt{\frac{|\SS^{d-1}|}{s}} \cdot \left(f^{(q)} - f\right)(w_j). \]
Additionally, for ease of notation let $y := f^{(q)} - f$.
Thus we now focus on bounding $\| \P \cdot \y \|_{\SS^{d-1}}^2$.
We start by computing the expectation of this quantity with respect to $w_1, w_2, \ldots, w_s$,
\begin{align}
    \E \left[ \| \P \cdot \y \|_{\SS^{d-1}}^2 \right] &= \E \left[ \left\| \sqrt{\frac{|\SS^{d-1}|}{s}} \cdot \sum_{j=1}^s \phi_{w_j} \cdot \y(j) \right\|_{\SS^{d-1}}^2 \right]\nonumber \\
    &= \frac{|\SS^{d-1}|^2}{s^2} \cdot \E \left[ \left\| \sum_{j=1}^s \phi_{w_j} \cdot y(w_j) \right\|_{\SS^{d-1}}^2 \right]\nonumber \\
    &= \frac{|\SS^{d-1}|^2}{s^2} \sum_{i,j \in [s]} \E_{w_i, w_j} \left[ \langle \phi_{w_i}, \phi_{w_j} \rangle_{\SS^{d-1}} \cdot y(w_i) y(w_j) \right]\nonumber\\
    &= \frac{|\SS^{d-1}|^2}{s^2} \sum_{i\in [s]} \E_{w_i} \left[ \|\phi_{w_i}\|_{\SS^{d-1}}^2 \cdot y(w_i)^2 \right] \label{norm-Py-diagonal}\\
    &\qquad + \frac{|\SS^{d-1}|^2}{s^2} \sum_{i \neq j \in [s]} \left< \E_{w_i} [ y(w_i) \cdot \phi_{w_i} ] , \E_{w_j} [ y(w_j) \cdot \phi_{w_j} ] \right>_{\SS^{d-1}}  \label{norm-Py-offdiagonal},
\end{align}
First we consider the term in \cref{norm-Py-diagonal}. By \cref{claim-gegen-kernel-properties} we can write,
\begin{align*}
    \frac{|\SS^{d-1}|^2}{s^2} \sum_{i\in [s]} \E_{w_i} \left[ \|\phi_{w_i}\|_{\SS^{d-1}}^2 \cdot y(w_i)^2 \right] &= \frac{|\SS^{d-1}|^2}{s^2} \sum_{i\in [s]} \E_{w_i} \left[ \frac{\beta_{q,d}}{|\SS^{d-1}|} \cdot y(w_i)^2 \right]\\
    &=\frac{\beta_{q,d}}{s} \cdot \| y \|_{\SS^{d-1}}^2.
\end{align*}
Next we consider the term in \cref{norm-Py-offdiagonal}. Using the definition of $y = f^{(q)} - f$, We show that for any $\sigma \in \SS^{d-1}$ and any $i \in [s]$,
\begin{align*}
    \E_{w_i} [ y(w_i) \phi_{w_i}(\sigma) ] &= \E_{w \sim \mathcal{U}(\SS^{d-1})} \left[ \sum_{\ell=0}^q \frac{\alpha_{\ell,d}}{|\SS^{d-1}|}  P_{d}^{\ell}\left( \langle \sigma , w \rangle \right) y(w) \right]\\
    &=\frac{1}{|\SS^{d-1}|} \cdot \left[ \Kc_d^{(q)} y \right](\sigma)\\
    &= \left[ \Kc_d^{(q)} (f^{(q)} - f) \right](\sigma)\\
    &= \left[ \Kc_d^{(q)} \left( \Kc_d^{(q)} f - f \right) \right](\sigma)\\
    &= 0,
\end{align*}
where the last line above follows from \cref{claim:k-square-eqaul-k}. Thus,
\[ \frac{|\SS^{d-1}|^2}{s^2} \sum_{i \neq j \in [s]} \left< \E_{w_i} [ y(w_i) \cdot \phi_{w_i} ] , \E_{w_j} [ y(w_j) \cdot \phi_{w_j} ] \right>_{\SS^{d-1}}  = 0 \]
By plugging these equalities into \cref{norm-Py-diagonal} and \cref{norm-Py-offdiagonal} we find that,
\[ \E \left[ \| \P \cdot \y \|_{\SS^{d-1}}^2 \right] = \frac{\beta_{q,d}}{s} \cdot \| y \|_{\SS^{d-1}}^2 = \frac{\beta_{q,d}}{s} \cdot \| f^{(q)} - f \|_{\SS^{d-1}}^2 = \frac{\beta_{q,d}}{s} \cdot B^* \]
Thus, by Markov's inequality and using the fact that $s = \Omega\left( \frac{ \beta_{q,d}}{\epsilon^2 \cdot \delta} \right)$, the claim follows.
\end{proofof}

\section{Efficient Algorithm for Spherical Harmonic Interpolation}\label{proofofefficientalgorithm}
In this section we prove our main theorem about our spherical harmonic interpolation algorithm.

\efficientsphericqalharmoicinterpolationtheorem*

\begin{proof}
First note that the random points $w_1, w_2, \ldots, w_s$ in line~\ref{alg-line-randompoints-w} of \cref{alg:leverage-sample-regression} are i.i.d. sample with uniform distribution on the surface of $\SS^{d-1}$. Therefore, we can invoke \cref{thm:discretizing-regression-leverage}. More specifically, if we let $\P$ be the quasi-matrix defined in \cref{thm:discretizing-regression-leverage} corresponding to the random points $w_1, w_2, \ldots, w_s$ sampled in line~\ref{alg-line-randompoints-w} and if we let $\f$ be the vector of function samples defined in line~\ref{alg-line-sampled-function} of the algorithm, then with probability at least $1 - \delta$, any optimal solution to the following least-squares problem 
\begin{equation}\label{eq:discretized-least-squares-problem}
    \tilde{g} \in \arg\min_{g \in L^2 \left( \SS^{d-1} \right)} \left\| \P^* g - \f \right\|_2^2,
\end{equation}
satisfies the following,
\begin{equation}\label{eq:approximate-regression-guarantee}
    \left\| \Kc^{(q)}_{d} \tilde{g} - f \right\|_{\SS^{d-1}}^2 \le (1+\epsilon) \cdot \min_{g \in L^2\left( \SS^{d-1} \right)} \left\| \Kc^{(q)}_{d} g - f \right\|_{\SS^{d-1}}^2.
\end{equation}
Now note that the least-squares problem in \cref{eq:discretized-least-squares-problem} has at least one optimal solution $\tilde{g}$ which is in the eigenspace of the operator $\P \P^*$. More specifically, there exists a vector $\z \in \RR^s$ such that $\tilde{g} = \P \cdot \z$ is an optimal solution for \cref{eq:discretized-least-squares-problem}. Therefore, we can focus on finding this optimal solution by solving the following least-squares problem 
\[\z \in \arg\min_{\x \in \RR^s} \left\| \P^* \P \x - \f \right\|_2^2,\] 
and then letting $\tilde{g} = \P \cdot \z$. This $\tilde{g}$ is guaranteed to be an optimal solution for \cref{eq:discretized-least-squares-problem}, thus it satisfies \cref{eq:approximate-regression-guarantee}. We solve the above least-squares problem using the kernel trick. In fact we show that $\P^* \P$ is equal to the kernel matrix $\K$ computed in line~\ref{alg-line-kernel-def} of \cref{alg:leverage-sample-regression}. To see why, note that for any $i , j \in [s]$ we have,
\begin{align*}
    \left[ \P^* \P \right]_{i,j} &= \left< \sum_{\ell=0}^q \frac{\alpha_{\ell,d}}{\sqrt{s \cdot |\SS^{d-1}|}} \cdot P_{d}^{\ell}\left( \langle w_i , \cdot \rangle \right) , \sum_{\ell=0}^q \frac{\alpha_{\ell,d}}{\sqrt{s \cdot |\SS^{d-1}|}} \cdot P_{d}^{\ell}\left( \langle w_j , \cdot \rangle \right) \right>_{\SS^{d-1}}\\
    &= \sum_{\ell=0}^q \sum_{\ell'=0}^q \frac{\alpha_{\ell,d} \alpha_{\ell',d}}{s \cdot |\SS^{d-1}|} \cdot \left< P_{d}^{\ell}\left( \langle w_i , \cdot \rangle \right), P_{d}^{\ell'}\left( \langle w_j , \cdot \rangle \right) \right>_{\SS^{d-1}}\\
    &= \sum_{\ell=0}^q \sum_{\ell'=0}^q \frac{\alpha_{\ell,d} \alpha_{\ell',d}}{s} \cdot \E_{v \sim \mathcal{U}(\SS^{d-1})} \left[ P_{d}^{\ell}\left( \langle w_i , v \rangle \right) \cdot P_{d}^{\ell'}\left( \langle w_j , v \rangle \right) \right]\\
    &= \sum_{\ell=0}^q \frac{\alpha_{\ell,d}}{s} \cdot P_{d}^{\ell}\left( \langle w_i , w_j \rangle \right) = \K_{i,j},
\end{align*}
where the fourth line above follows from \cref{claim-gegen-kernel-properties}. Therefore, we are interested in the optimal solution of the following least-squares problem 
\[\z \in \arg\min_{\x \in \RR^s} \left\| \K \x - \f \right\|_2^2.\] 
The least-squares solution to the above problem is $\z = \K^{\dagger} \f$ which is exactly what is computed in line~\ref{alg-line-opt-solution-regression} of the algorithm. Now note that, the function $\tilde{g} = \P \cdot \z$ satisfies \cref{eq:approximate-regression-guarantee}. Because $\tilde{g} = \P \cdot \z \in \Hc^{(q)}(\SS^{d-1})$ and because $\Kc^{(q)}_d$ is an orthonormal projection operator into $\Hc^{(q)}(\SS^{d-1})$, we have $\Kc^{(q)}_d \cdot \tilde{g} = \tilde{g} = \P \cdot \z$. This together with \cref{eq:approximate-regression-guarantee} imply that,
\[ \left\| \P \cdot \z - f \right\|_{\SS^{d-1}}^2 \le (1+\epsilon) \cdot \min_{g \in L^2\left( \SS^{d-1} \right)} \left\| \Kc^{(q)}_{d} g - f \right\|_{\SS^{d-1}}^2. \]
Now if we invoke \cref{claim-interpolation-via-least-squares} with $C = 1+\epsilon$ on the above inequality we find that,
\[ \left\| \P \cdot \z - f^{(q)} \right\|_{\SS^{d-1}}^2 \le \epsilon \cdot \left\| f^{(q)} - f \right\|_{\SS^{d-1}}^2. \]
Finally, one can easily see that the function $y \in \Hc^{(q)}(\SS^{d-1})$ that \cref{alg:leverage-sample-regression} outputs in line~\ref{alg-line-output-function} is exactly equal to $y = \P \cdot \z$. This completes the accuracy
bound of the theorem. 

\paragraph{Runtime and Sample Complexity.} these bounds follow from observing that:
\begin{itemize}
    \item $s\cdot d$ time is needed to generate $w_1,w_2, \ldots, w_s$ in line~\ref{alg-line-randompoints-w} of the algorithm. To do this, we first generate random Gaussian points in $\RR^d$ and then project then onto $\SS^{d-1}$ by normalizing them.
    \item $s^2 \cdot d$ operations are needed to form the kernel matrix $\K$ in line~\ref{alg-line-kernel-def} of the algorithm.
    \item $s$ queries to function $f$ are needed to form the samples vector $\f$ in line~\ref{alg-line-sampled-function} of the algorithm.
    \item $s^\omega$ time is needed to compute the least-squares solution $\z = \K^{\dagger} \f$ in line~\ref{alg-line-opt-solution-regression} of the algorithm.
    \item $s \cdot d$ operations are needed to evaluate the output function $y(\sigma)$ in line~\ref{alg-line-output-function} of the algorithm.
\end{itemize}
This completes the proof of \cref{thm:efficient-regression-alg}.
\end{proof}

\section{Lower Bound: Claims and Lemmas}\label{appndx-lower-bound}
In this section we prove the Claims and Lemmas used in our lower bound analysis for proving \cref{thm:lower-bound}.

\claimlowerboundhardinputneedszeroregressionerror*
\begin{proof}
Note that \cref{interpolation-problem} requires recovering a function $\tilde{f}^{(q)} \in \Hc^{(q)}\left( \SS^{d-1} \right)$ such that: 
\begin{equation}\label{eq:problem1-guarantee}
    \left\| \tilde{f}^{(q)} - f^{(q)} \right\|_{\SS^{d-1}}^2 \le \epsilon \cdot \left\| f^{(q)} - f \right\|_{\SS^{d-1}}^2, 
\end{equation}
where $f^{(q)} = \Kc^{(q)}_{d} f$.
Using the definition of the input function $f = \sum_{\ell=0}^q \Y_{\ell} \cdot v^{(\ell)}$, we can write,
\begin{align*}
    f^{(q)} = \Kc^{(q)}_{d} f &= \sum_{\ell=0}^q \Kc^{(q)}_{d} \cdot \Y_{\ell} \cdot v^{(\ell)}\\
    &= \sum_{\ell=0}^q \left(\sum_{\ell'=0}^q\Y_{\ell'} \Y_{\ell'}^* \right) \cdot \Y_{\ell} \cdot v^{(\ell)} \\
    &= \sum_{\ell=0}^q \Y_{\ell} \cdot v^{(\ell)} = f,
\end{align*}
where the equality in the second line above follows from \cref{eq:eigendecomposition-kernel-operator} and the addition theorem in \cref{lem:additive-formula}, and the third line follows because the operator $\Y_\ell$ has orthonormal columns and thus $\Y_{\ell'}^* \Y_{\ell} = I_{\alpha_{\ell,d}} \cdot \mathbbm{1}_{\{\ell = \ell'\}}$. Therefore, plugging this into \cref{eq:problem1-guarantee} gives,
\[ \left\| \tilde{f}^{(q)} - f \right\|_{\SS^{d-1}}^2 = \left\| \tilde{f}^{(q)} - f^{(q)} \right\|_{\SS^{d-1}}^2 \le \epsilon \cdot \|f^{(q)} - f\|_{\SS^{d-1}}^2 = \epsilon \cdot \|f - f\|_{\SS^{d-1}}^2 = 0.\]

\end{proof}

\lemlowerboundhardinstancecoeffsmustberecoverable*

\begin{proof}
By \cref{problem1-guarantee-lowerbound}, the output of the algorithm that solves \cref{interpolation-problem}, satisfies $\left\| \tilde{f}^{(q)} - f \right\|_{\SS^{d-1}}^2 =0$. Therefore, by orthonormality of the columns of the operator $\Y_{\ell}$, we can write,
\[ \Y_{\ell}^* \tilde{f}^{(q)} = \Y_{\ell}^* f + \Y_{\ell}^* (\tilde{f}^{(q)} - f) = \sum_{\ell'=0}^q \Y_{\ell}^* \Y_{\ell'} \cdot v^{(\ell')} = v^{(\ell)}. \]

\end{proof}

\end{document}